\documentclass[11pt,english]{smfart}
\usepackage[T1]{fontenc}
\usepackage[colorlinks=true]{hyperref}
\hypersetup{urlcolor=blue, linkcolor=blue, citecolor=red, anchorcolor=blue]}
\usepackage{amssymb,url,xspace,smfthm}
\usepackage{amsmath}
\usepackage{paralist}
\usepackage{graphics}
\usepackage{epsfig}
\usepackage{graphicx}
\usepackage[bb=boondox]{mathalfa}
\usepackage{epstopdf}
\usepackage{setspace}
\usepackage{fourier,mathrsfs}
\usepackage[normalem]{ulem}

\newcounter{Hequation}

\makeatletter\g@addto@macro\equation{\stepcounter{Hequation}}\makeatother

\newtheorem{theorem}{Theorem}
\newtheorem{corollary}[theorem]{Corollary}
\newtheorem{lemma}[theorem]{Lemma}
\newtheorem{proposition}[theorem]{Proposition}

\theoremstyle{definition}

\newcommand\N{{\mathbb N}}
\newcommand\R{{\mathbb R}}
\newcommand\T{{\mathbb T}}

\newcommand\Z{{\mathbb Z}}

\newcommand{\RR}{\mathbb{R}}

\newcommand{\cD}{\mathcal D}

\def\sfL{\mathsf{L}}
\def\sfT{\mathsf T}

\DeclareMathOperator{\re}{Re}

\def\eps{{\varepsilon}}

\let\oldmarginpar\marginpar
\renewcommand\marginpar[1]{\-\oldmarginpar[\raggedleft\footnotesize #1]%
{\raggedright\footnotesize #1}}

\newcommand{\nrm}[2]{\left\|{#1}\right\|_{#2}}

\newcommand{\nc}{\normalcolor}

\newcommand{\Email}[1]{\email{\href{mailto:#1}{\textsf{#1}}}}
\newcommand{\be}[1]{\begin{equation}\label{#1}}
\newcommand{\ee}{\end{equation}}
\renewcommand{\(}{\left(}
\renewcommand{\)}{\right)}
\newcommand{\fav}{{f_{\kern-0.5pt\bullet}}}

\title[Hypocoercivity without confinement]{Hypocoercivity without confinement}
\date{\today}
\author[E.~Bouin]{Emeric Bouin}
\address[E.~Bouin]{CEREMADE (CNRS UMR n$^\circ$ 7534), PSL university, Universit\'e Paris-Dauphine, Place de Lattre de Tassigny, 75775 Paris 16, France}
\Email{bouin@ceremade.dauphine.fr}

\author[J.~Dolbeault]{Jean Dolbeault}
\address[J.~Dolbeault]{CEREMADE (CNRS UMR n$^\circ$ 7534), PSL university, Universit\'e Paris-Dauphine, Place de Lattre de Tassigny, 75775 Paris 16, France}
\Email{dolbeaul@ceremade.dauphine.fr}

\author[S.~Mischler]{St\'ephane Mischler}
\address[S.~Mischler]{CEREMADE (CNRS UMR n$^\circ$ 7534), PSL university, Universit\'e Paris-Dauphine, Place de Lattre de Tassigny, 75775 Paris 16, France}
\Email{mischler@ceremade.dauphine.fr}

\author[C.~Mouhot]{\newline\hspace*{-1cm} Cl\'ement Mouhot}
\address[C.~Mouhot]{DPMMS, Center for Mathematical Sciences, University of Cambridge, Wilberforce Road, Cambridge CB3 0WA, UK}
\Email{C.Mouhot@dpmms.cam.ac.uk}

\author[C.~Schmeiser]{Christian Schmeiser}
\address[C.~Schmeiser]{Fakult\"at f\"ur Mathematik, Universit\"at Wien, Oskar-Morgenstern-Platz 1, 1090 Wien, Austria}
\Email{Christian.Schmeiser@univie.ac.at}

\thanks{Corresponding author: \'Emeric Bouin}

\begin{abstract} In this paper, hypocoercivity methods are applied to linear kinetic equations with mass conservation and without confinement, in order to prove that the solutions have an algebraic decay rate in the long-time range, which the same as the rate of the heat equation. Two alternative approaches are developed: an analysis based on decoupled Fourier modes and a direct approach where, instead of the Poincar\'e inequality for the Dirichlet form, Nash's inequality is employed. The first approach is also used to provide a simple proof of exponential decay to equilibrium on the flat torus. The results are obtained on a space with exponential weights and then extended to larger function spaces by a factorization method. The optimality of the rates is discussed. Algebraic rates of decay on the whole space are improved when the initial datum has moment cancellations. \end{abstract}
\begin{document}
\maketitle
\vspace*{-1cm}

%%%%%%%%%%%%%%%%%%%%%%%%%%%%%%%%%%%%%%%%%%%%%%%%%%%%%%%%%%%%%%%%%%%%%%
%%%%%%%%%%%%%%%%%%%%%%%%%%%%%%%%%%%%%%%%%%%%%%%%%%%%%%%%%%%%%%%%%%%%%%
\section{Introduction}\label{sec:introduction}

We consider the Cauchy problem
\be{eq:model}
\partial_t f+v\cdot\nabla_xf=\sfL f\,,\quad f(0,x,v)=f_0(x,v)
\ee
for a distribution function $f(t,x,v)$, with \emph{position} variable $x\in\R^d$, \emph{velocity} variable $v\in\RR^d$, and with \emph{time} $t\ge 0$. Concerning the \emph{collision operator} $\sfL$, we shall consider two cases:
\begin{itemize}
\item[(a)] \emph{Fokker-Planck} collision operator:
\[
\sfL f=\nabla_v\cdot\Big[M\,\nabla_v\(M^{\kern 0.5pt-1}\,f\)\Big]\,,
\]
\item[(b)] \emph{Scattering} collision operator:
\[
\sfL f=\int_{\R^d}\sigma(\cdot,v')\,\big(f(v')\,M(\cdot)-f(\cdot)\,M(v')\big)\,dv'\,.
\]
\end{itemize}
We shall make the following assumptions on the \emph{local equilibrium} $M(v)$ and on the \emph{scattering rate} $\sigma(v,v')$:
\begin{align}
\tag{H1}\label{H1}&\int_{\R^d} M(v)\,dv = 1 \,, \quad \nabla_v\sqrt{M} \in \mathrm L^2(\R^d) \,,\quad M\in C(\R^d)\,,\\
\nonumber&\hspace*{1cm}M = M(|v|)\,,\quad 0 < M(v) \le c_1 e^{-c_2|v|} \,,\quad \forall\,v\in\R^d \,, \quad\text{for some } c_1,\;c_2 > 0 \,.\\
\tag{H2}\label{H2}&1\le \sigma(v,v') \le \overline\sigma \,, \quad\forall v,v'\in\R^d \,,\quad \text{for some } \overline\sigma \ge 1 \,.\\
\tag{H3}\label{H3}&\int_{\R^d}\big(\sigma(v,v')-\sigma(v',v)\big)\,M(v')\,dv'=0,
                    \quad\forall\,v\in\R^d \,.
\end{align}

Before stating our main results, let us list some preliminary observations.
\\
(i) A typical example of a \emph{local equilibrium}
satisfying~\eqref{H1} is the Gaussian 
\be{Gaussian}
M(v)=\frac{e^{-\frac{|v|^2}2}}{(2\pi)^{d/2}}  \,.
\ee
(ii) With $\sigma\equiv1$, Case~(b) includes the relaxation operator \hbox{$\sfL f=M\rho_f-f$}, also known as the \emph{linear BGK operator}, with position density defined by
\[
\rho_f(t,x):=\int_{\R^d}f(t,x,v)\,dv\,.
\]
(iii) Positivity and exponential decay of the local equilibrium are essential for our approach. The assumption on the gradient and
continuity are technical and only needed for some of our results. Rotational symmetry is not important, but assumed for 
computational convenience. However the property
$$
\int_{\R^d} vM(v)dv = 0 \,,
$$
\emph{i.e.}, \emph{zero flux in local equilibrium}, is essential. 
\\
(iv) Since micro-reversibility (or detailed balance), \emph{i.e.}, symmetry of $\sigma$, is not required, Assumption~\eqref{H3} is needed for \emph{mass conservation}, \emph{i.e.}, 
$$
\int_{\R^d} \sfL f\,dv = 0 \,, 
$$
in Case (b). The boundedness away from zero of $\sigma$ in~\eqref{H2}
guarantees coercivity of~$\sfL$ relative to its nullspace (such bound
can always be written $\sigma \ge 1$ by scaling).

\medskip Since $e^{t \sfL}$ propagates probability densities,
\emph{i.e.}, conserves mass and nonnegativity, $\sfL$ dissipates
convex relative entropies, implying in particular
$$
\int_{\R^d} \sfL f \frac{f}{M}\,dv \le 0 \,.
$$
This suggests to use the $\mathrm L^2$-space with the measure
$d\gamma_\infty := \gamma_\infty\,dv$, where
$\gamma_\infty(v)=M(v)^{-1}$, as a functional analytic framework (the
subscript $\infty$ will make sense later). We shall need the
\emph{microscopic coercivity} property \be{H4}\tag{H4} -\int_{\R^d}
f\,\mathsf{L}f\,d\gamma_\infty \ge
\lambda_m\int_{\R^d}\(f-M\,\rho_f\)^2\,d\gamma_\infty \,, \ee with
some $\lambda_m>0$. In Case~(a) it is equivalent to the Poincar\'e
inequality with weight~$M$,
\[\label{eq:Poincarea}
  \int_{\R^d}|\nabla_v h|^2\,M\,dv\ge\lambda_m
  \int_{\R^d}\(h-\int_{\R^d} h\,M\,dv\)^2\,M\,dv \,,
\]
for all $h=f/M\in\mathrm{H}^1(M\,dv)$. It holds as a consequence of
the exponential decay assumption in~\eqref{H1} (see, \emph{e.g.},
\cite{Nash,bakry:hal-00634523}). For the normalized Gaussian
\eqref{Gaussian} the optimal constant is known to be $\lambda_m=1$ (see for instance~\cite{MR954373} and references therein).
In Case~(b), \eqref{H4} means
\[\label{eq:Poincareb}
\frac12 \iint_{\R^d\times\R^d}\sigma(v,v')\,M(v)\,M(v')\(u(v)-u(v')\)^2dv'\,dv\ge\lambda_m\int_{\R^d}\(u-\rho_{u\,M}\)^2M\,dv \,,
\]
for all $u=f/M\in\mathrm L^2(M\,dv)$, and it holds with $\lambda_m=1$
as a consequence of the lower bound for $\sigma$ in
Assumption~\eqref{H2}.

Although the transport operator does not contribute to entropy
dissipation, its dispersion in the $x$-direction in combination with
the dissipative properties of the collision operator yields the
desired decay results.  In order to perform a \emph{mode-by-mode
  hypocoercivity} analysis, we introduce the Fourier representation
with respect to $x$,
\[
  f(t,x,v)=\int_{\R^d}\hat f(t,\xi,v)\,e^{+i\,x\cdot\xi}\,d\mu(\xi)\,,
\]
where $d\mu(\xi)=(2\pi)^{-d}\,d\xi$ and $d\xi$ is the Lesbesgue
measure on $\R^d$. The normalization of $d\mu(\xi)$ is chosen such
that Planche\-rel's formula reads
\[
  \left\|f(t,\cdot,v)\right\|_{\mathrm L^2(dx)}=
  \left\|\hat f(t,\cdot,v)\right\|_{\mathrm L^2(d\mu(\xi))}
\]
with a straightforward abuse of notations. The Cauchy
problem~\eqref{eq:model} in Fourier variables is now decoupled in the
$\xi$-direction: \be{eq:hat} \partial_t \hat f+i\,(v\cdot\xi)\,\hat
f=\sfL\hat f\,,\quad\hat f(0,\xi,v)=\hat{f_0}(\xi,v) \,.  \ee

Our main results are devoted to \emph{hypocoercivity without
  confinement}: when the variable $x$ is taken in $\R^d$, we assume
that there is no potential preventing the runaway corresponding to
$|x|\to+\infty$. So far, hypocoercivity results have been obtained
either in the compact case corresponding to a bounded domain in~$x$,
for instance $\T^d$, or in the whole Euclidean space with an external
potential $V$ such that the measure $e^{-V}\,dx$ admits a Poincar\'e
inequality. Usually other technical assumptions are required on $V$
and there are many variants (for instance one can assume a stronger
logarithmic Sobolev inequality instead of a Poincar\'e inequality),
but the common property is that some growth condition on $V$ is
assumed and in particular the measure $e^{-V}\,dx$ is bounded. Here we
consider the case $V\equiv0$, which is obviously a different
regime. By replacing the Poincar\'e inequality by Nash's inequality or
using direct estimates in Fourier variables, we adapt the
$\mathrm L^2$ hypocoercivity methods and prove that an appropriate
norm of the solution decays at a rate which is the rate of the heat
equation. This observation is compatible with diffusion limits, which
have been a source of inspiration for building Lyapunov functionals
and establishing the $\mathrm L^2$ hypocoercivity method
of~\cite{DMS-2part}. Before stating any result, we need some notation
to implement the \emph{factorization} method of~\cite{MR3779780} and
obtain estimates in large functional spaces.

Let us consider the measures \be{Def:Measures}
d\gamma_k:=\gamma_k(v)\,dv\quad\mbox{where}\quad\gamma_k(v)=\left(1+|v|^2\right)^{k/2}
\quad\mbox{and}\quad k>d\,, \ee such that
$1/\gamma_k\in\mathrm L^1(\R^d)$. The condition $k\in(d,\infty]$ then
covers the case of weights with a growth of the order of $|v|^k$, when
$k$ is finite, and we denote $k=\infty$ the case when the weight
$\gamma_\infty = M^{-1}$ grows at least exponentially fast.

%---------------------------------------------------------------------
\begin{theorem}\phantomsection\label{th:whole-space-mode} Assume~\eqref{H1}--\eqref{H4}, $x\in\R^d$, and $k\in(d,\infty]$. Then there exists a constant $C>0$ such that solutions $f$ of~\eqref{eq:model} with initial datum $f_0\in\mathrm L^2(dx\,d\gamma_k)\cap\mathrm L^2(d\gamma_k;\,\mathrm L^1(dx))$ satisfy, for all $t\ge 0$,
\[
\left\|f(t,\cdot,\cdot)\right\|_{\mathrm L^2\(dx\,d\gamma_k\)}^2\le C\,\frac{\left\|f_0\right\|_{\mathrm L^2\(dx\,d\gamma_k\)}^2+ \left\|f_0\right\|^2_{\mathrm L^2\(d\gamma_k;\,\mathrm L^1(dx)\)}}{(1+t)^{d/2}}\,.
\]
\end{theorem}
%---------------------------------------------------------------------
For the heat equation improved decay rates can be shown by Fourier techniques, if the modes with slowest decay are eliminated
from the initial data. The following two results are in this spirit. 
%-----------------------------------------------------------------------------
\begin{theorem}\phantomsection\label{th:whole-space-zero-average} 
Let the assumptions of Theorem \ref{th:whole-space-mode} hold, and let 
$$
\iint_{\R^d\times \R^d}f_0\,dx\,dv=0 \,.
$$
Then there exists $C>0$ such that solutions $f$ of~\eqref{eq:model} with initial datum $f_0$ satisfy, for all $t\ge 0$,
\[
\left\|f(t,\cdot,\cdot)\right\|_{\mathrm L^2(dx\,d\gamma_k)}^2 \le C\,\frac{\left\|f_0\right\|_{\mathrm L^2(d\gamma_{k+2};\,\mathrm L^1(dx))}^2+\left\|f_0\right\|^2_{\mathrm L^2\(d\gamma_k;\,\mathrm L^1(|x|\,dx)\)}+\left\|f_0\right\|^2_{\mathrm L^2(dx\,d\gamma_k)}}{(1+t)^{d/2+1}},
\]
with $k \in(d,\infty)$.
\end{theorem}
%-------------------------------------------------------------------------------

The case of Theorem~\ref{th:whole-space-zero-average}, but with $k=\infty$, is covered in Theorem~\ref{th:whole-space-zero-average2} under the stronger assumption that $M$ is a Gaussian. For the formulation of a result corresponding to the cancellation of higher order moments, we introduce the set $\R_\ell[X,V]$ of polynomials of order at most $\ell$ in the variables $X$, $V\in\R^d$ (the sum of the degrees in $X$ and in $V$ is at most $\ell$). We also need that the kernel of the collision operator is spanned by a Gaussian function in order to keep polynomial spaces invariant. This means that for any $P \in \R_\ell[X,V]$, one has $\( \sfL - \sfT \)(PM) \in \R_\ell[X,V] M$. Since the transport operator mixes both variables $x$ and $v$, one needs moments with respect to both $x$ and $v$ variables.
%---------------------------------------------------------------------
\begin{theorem}\phantomsection\label{th:whole-space-zero-average2} In Case (a), let $M$ be the normalized Gaussian \eqref{Gaussian}. In Case (b), we assume that $\sigma\equiv 1$. Let $k\in(d,\infty]$, $\ell\in\N$ and assume that the initial datum $f_0\in\mathrm L^1(\R^d\times\R^d)$ is such that
\be{ZeroAverage}\iint_{\R^d\times \R^d}f_0(x,v)\, P(x,v)\,dx\,dv=0
\ee
for all $P \in \R_\ell[X,V]$. Then there exists a constant $c_k>0$ such that any solution $f$ of~\eqref{eq:model} with initial datum $f_0$ satisfies, for all $t\ge 0$,
\[
\left\|f(t,\cdot,\cdot)\right\|_{\mathrm L^2(dx\,d\gamma_k)}^2\le c_k\,\frac{\left\|f_0\right\|_{\mathrm L^2(d\gamma_{k+2};\,\mathrm L^1(dx))}^2+\left\|f_0\right\|^2_{\mathrm L^2\(d\gamma_k;\,\mathrm L^1(|x|\,dx)\)}+\left\|f_0\right\|^2_{\mathrm L^2(dx\,d\gamma_k)}}{(1+t)^{d/2+1 + \ell}}.
\]
\end{theorem}
%--------------------------------------------------------------------- 

The {\bf outline of this paper} goes as follows. In Section~\ref{sec:mode-mode-hypoc}, we slightly strengthen the \emph{abstract hypocoercivity} result of~\cite{DMS-2part} by allowing complex Hilbert spaces and by providing explicit formulas for the coefficients in the decay rate (Proposition~\ref{theo:DMS}). In Corollary~\ref{lec14-modelemm}, this result is applied for fixed $\xi$ to the Fourier transformed problem~\eqref{eq:hat}, where integrals are computed with respect to the measure $d\gamma_\infty$ in the velocity variable~$v$. Since the frequency $\xi$ can be considered as a parameter, we shall speak of a \emph{mode-by-mode hypocoercivity} 
result. It provides exponential decay, however with a rate deteriorating as $\xi\to 0$.

In Section~\ref{sec:factorization}, we state a special case (Proposition~\ref{th:factorization}) of the \emph{factorization} result of~\cite{MR3779780} with explicit constants which corresponds to an \emph{enlargement} of the space, and also a \emph{shrinking} result (Proposition~\ref{prop:shrink}) which will be useful in Section~\ref{Sec:ImprovedDecayRates}. By the enlargement result, the estimate corresponding to the exponential weight $\gamma_\infty$ is extended in Corollary~\ref{th:fac-appl} to larger spaces corresponding to the algebraic weights~$\gamma_k$ with $k\in(d,\infty)$. As a straightforward consequence, in Section~\ref{Sec:AsymptoticFourier}, we recover an \emph{exponential convergence rate} in the case of the flat torus~$\T^d$ (Corollary~\ref{Thm:torus}), and then give a first proof of the \emph{algebraic decay rate} of Theorem~\ref{th:whole-space-mode} in the whole space without confinement.

In Section~\ref{Sec:WholeSpaceNash}, an hypocoercivity method, where the Poincar\'e inequality, or the so-called \emph{macroscopic coercivity} condition, is replaced by the \emph{Nash inequality}, provides an alternative proof of Theorem~\ref{th:whole-space-mode}. Such a direct approach is also applicable to problems with non-constant coefficients like scattering operators with $x$-dependent scattering rates $\sigma$, or Fokker-Planck operators with $x$-dependent diffusion constants like $\nabla_v\cdot\big(\mathcal D(x)\,M\,\nabla_v(M^{\kern 0.5pt-1}\,f)\big)$.

The \emph{improved algebraic decay rates} of Theorem~\ref{th:whole-space-zero-average}
and Theorem~\ref{th:whole-space-zero-average2} are obtained by direct Fourier estimates in Section~\ref{Sec:Fourier}. As we shall see in the~\ref{Sec:Green}, the rates of Theorem~\ref{th:whole-space-mode} are optimal: the decay rate is the rate of the heat equation on $\R^d$. Our method is consistent with the \emph{diffusion limit} and provides estimates which are asymptotically uniform in this regime: see~\ref{sec:consistency}. We also check that the results of Theorem~\ref{th:whole-space-zero-average} and Theorem~\ref{th:whole-space-zero-average2} are uniform in the diffusive limit in \ref{sec:consistency}.\smallskip

We conclude this introduction by a brief {\bf review of the literature:} On the whole Euclidean space, we refer to~\cite{2017arXiv170610034V} for recent lecture notes on available techniques for capturing the large time asymptotics of the heat equation. Some of our results make a clear link with the heat flow seen as the diffusion limit of the kinetic equation. We also refer to~\cite{Iacobucci2017} for recent results on the diffusion limit, or overdamped limit (see \ref{sec:consistency}).

The mode-by-mode analysis is an extension of the hypocoercivity theory of~\cite{DMS-2part}, which has been inspired by~\cite{MR2215889}, but is also close to the Kawashima compensating function method: see~\cite{MR1057534} and \cite[Chapter 3, Section~3.9]{MR1379589}. We also refer to~\cite{Duan2011} where the Kawashima approach is applied to a particular case of the scattering model~(b).

The word \emph{hypocoercivity} was coined by T.~Gallay and widely
disseminated in the context of kinetic theory by
C.~Villani. In~\cite{Mouhot-Neumann,MR2275692,Mem-villani}, the method
deals with large time properties of the solutions by considering a
$\mathrm H^1$-norm (in $x$ and~$v$ variables) and taking into account
cross-terms. This is very well explained
in~\cite[Section~3]{MR2275692}, but was already present in earlier
works like~\cite{MR2034753}. Hypocoercivity theory is inspired by and
related to the earlier \emph{hypoellipticity} theory. The latter has a
long history in the context of the kinetic Fokker-Planck equation. One
can refer for instance to~\cite{MR1969727,MR2034753} and much earlier
to H\"ormander's theory~\cite{MR0222474}. The seed for such an
approach can even be traced back to Kolmogorov's computation of
Green's kernel for the kinetic Fokker-Planck equation
in~\cite{MR1503147}, which has been reconsidered in~\cite{MR0168018}
and successfully applied, for instance, to the study of the
Vlasov-Poisson-Fokker-Planck system in~\cite{MR1052014,MR1200643}.

Linear Boltzmann equations and BGK (Bhatnagar-Gross-Krook, see~\cite{bhatnagar1954model}) models also have a long history: we refer to~\cite{DegGouPou,Caceres-Carrillo-Goudon} for key mathematical properties, and to~\cite{Mouhot-Neumann,MR2215889} for first hypocoercivity results. In this paper we will mostly rely on~\cite{Dolbeault2009511,DMS-2part}. However, among more recent contributions, one has to quote~\cite{MR3479064,achleitner2016linear,BouinHoffmann} and also an approach based on the Fisher information which has recently been implemented in~\cite{2017arXiv170204168E,2017arXiv170310504M}.

With the \emph{exponential weight} $\gamma_\infty=M^{\kern 0.5pt-1}$, Corollary~\ref{Thm:torus} can be obtained directly by the method of~\cite{DMS-2part}. In this paper we also obtain a result for weights with polynomial growth in the velocity variable based on~\cite{MR3779780}. For completeness, let us mention that recently the exponential growth issue was overcome for the Fokker-Planck case in~\cite{Kavian,MisMou} by a different method. The improved decay rates established in Theorem~\ref{th:whole-space-zero-average} and in Theorem~\ref{th:whole-space-zero-average2} generalize to kinetic models similar results known for the heat equation, see for instance~\cite[Remark~3.2~(7)]{MisMou} or~\cite{Bartier201176}.

%%%%%%%%%%%%%%%%%%%%%%%%%%%%%%%%%%%%%%%%%%%%%%%%%%%%%%%%%%%%%%%%%%%%%%
%%%%%%%%%%%%%%%%%%%%%%%%%%%%%%%%%%%%%%%%%%%%%%%%%%%%%%%%%%%%%%%%%%%%%%
\section{Mode-by-mode hypocoercivity}
\label{sec:mode-mode-hypoc}

Let us consider the evolution equation
\be{EqnEvol}
\frac{dF}{dt}+\mathsf TF=\mathsf{L}F\,,
\ee
where $\mathsf T$ and $\mathsf{L}$ are respectively a general \emph{transport operator} and a general \emph{linear collision operator}. We shall use the abstract approach of~\cite{DMS-2part}. Although the extension of the method to Hilbert spaces over complex numbers is rather straightforward, we carry it out here for completeness. For details on the Cauchy problem or, \emph{e.g.}, on the domains of the operators, we refer to~\cite{DMS-2part}. Notice that we do not ask that $\mathsf L$ is a Hermitian operator but simply assume that $\mathsf L^*\mathsf A=0$.
%---------------------------------------------------------------------
\begin{proposition}\phantomsection\label{theo:DMS} Let $\mathsf{L}$
  and $\mathsf T$ be closed unbounded linear operators on the complex
  Hilbert space $(\mathcal{H},\langle\cdot,\cdot\rangle)$ with dense
  domains $\cD(L)$ and $\cD(T)$. Assume that $\mathsf T$ is
  anti-Hermitian. Let $\mathsf\Pi$ be the orthogonal projection onto
  the null space of \,$\mathsf{L}$ and define
  \[
    \mathsf{A}:=\big(1+(\mathsf T\mathsf\Pi)^*\mathsf
    T\mathsf\Pi\big)^{-1}(\mathsf T\mathsf\Pi)^*
  \]
  where ${}^*$ denotes the adjoint with respect to
  $\langle\cdot,\cdot\rangle$. We assume that $\mathsf L^*\mathsf A=0$
  and that there are \nc positive constants $\lambda_m$, $\lambda_M$,
  and $C_M$ exist, such that, for any $F\in\mathcal{H}$, the following
  properties hold:\\
  $\rhd$ \emph{microscopic coercivity:} \be{A1}
  -\,\langle\mathsf{L}F,F\rangle\ge\lambda_m\,\|(1-\mathsf\Pi)F\|^2,
  \quad \forall \, F \in \cD(L)\,,\tag{A1} \ee $\rhd$
  \emph{macroscopic coercivity:} \be{A2} \|\mathsf T\mathsf\Pi
  F\|^2\ge\lambda_M\,\|\mathsf\Pi F\|^2, \quad \forall \, F \in
  \cD(T)\,,\tag{A2} \ee $\rhd$ \emph{parabolic macroscopic dynamics:}
  \be{A3} \mathsf\Pi\mathsf T\mathsf\Pi\,F=0, \quad \forall \, F \in
  \cD(T)\,,\tag{A3} \ee $\rhd$ \emph{bounded auxiliary operators:}
  \be{A4} \|\mathsf{AT}(1-\mathsf\Pi)F\|+\|\mathsf{AL}F\|\le
  C_M\,\|(1-\mathsf\Pi)F\|, \quad \forall \, F \in \cD(L) \cap
  \cD(T)\,.\tag{A4} \ee

  Then $L-T$ generates a $C_0$-semigroup and for any $t\ge0$, we have
  \be{decay-const} \left\|e^{(\mathsf{L}-\,\mathsf
      T)\,t}\right\|^2\le3\,e^{-\lambda\,t}\quad\mbox{where}\quad\lambda=\frac{\lambda_M}{3\,(1+\lambda_M)}\min\left\{1,\lambda_m,\frac{\lambda_m\,\lambda_M}{(1+\lambda_M)\,C_M^2}\right\}\,.
  \ee
\end{proposition}
%---------------------------------------------------------------------

\begin{proof} For some $\delta>0$ to be determined later, the Lyapunov functional
\[
\mathsf{H}[F]:=\tfrac12\,\|F\|^2+\delta\,\re\langle\mathsf{A}F,F\rangle
\]
is such that $\frac d{dt}\mathsf{H}[F]=-\,\mathsf{D}[F]$ if $F$ solves~\eqref{EqnEvol}, with
\[
  \mathsf{D}[F]:=-\,\langle\mathsf{L}F,F\rangle+\delta\,\langle\mathsf{AT}\mathsf\Pi
  F,F\rangle+\delta\,\re\langle\mathsf{AT}(1-\mathsf\Pi)F,F\rangle
  -\delta\,\re\langle\mathsf{TA}F,F\rangle
  -\,\delta\,\re\langle\mathsf{AL}F,F\rangle\,.
\]
Note that we have used the fact that
$\re\langle\mathsf{A}F,\mathsf{L}F\rangle = 0$ because of the
assumption $\mathsf L^*\mathsf A=0$, and also that
$\langle\mathsf{AT}\mathsf\Pi F,F\rangle$ is real because
$\mathsf{AT}\mathsf\Pi$ is self-adjoint by construction. Since the
Hermitian operator $\mathsf{AT}\mathsf\Pi$ can be interpreted as the
application of the map $z\mapsto(1+z)^{-1}\,z$ to
$(\mathsf T\mathsf\Pi)^*\mathsf T\mathsf\Pi$ and as a consequence of
the spectral theorem~\cite[Theorem~VII.2, p.~225]{MR751959}, the
conditions~\eqref{A1} and~\eqref{A2} imply that
\[
  -\,\langle\mathsf{L}F,F\rangle
  +\delta\,\langle\mathsf{AT}\mathsf\Pi F,F\rangle\ge\lambda_m\,
  \|(1-\mathsf\Pi)F\|^2
  +\frac{\delta\,\lambda_M}{1+\lambda_M}\,\|\mathsf\Pi F\|^2\,.
\]
As in~\cite[Lemma~1]{DMS-2part}, if $G=\mathsf{A}F$, \emph{i.e.},
$G+(\mathsf T\mathsf\Pi)^*\mathsf T\mathsf\Pi\,G=(\mathsf
T\mathsf\Pi)^*F$, one has
\[
  \|\mathsf AF\|^2+\|\mathsf{TA}F\|^2=
  \langle G,G+(\mathsf T\mathsf\Pi)^*\,\mathsf T\mathsf\Pi\,G\rangle
  =\langle G,(\mathsf T\mathsf\Pi)^*F\rangle
  =\langle\mathsf{TA}F,\mathsf{(1-\Pi)}F\rangle
\]
where we have used $A = \Pi A$ and $\Pi T\Pi=0$. Using
$\left\vert \langle\mathsf{TA}F,\mathsf{(1-\Pi)}F\rangle \right\vert
\leq \|\mathsf{TA}F\|^2 + \frac14\,\|\mathsf{(1-\Pi)}F\|^2$,
one gets \be{A-bound} \|\mathsf AF\|^2
%+\frac12\,\|\mathsf{TA}F\|^2
\leq \frac14\,\|\mathsf{(1-\Pi)}F\|^2\,,
\ee which implies that
$\left|\re\langle\mathsf{A}F,F\rangle\right|\le \| AF \| \| F \| \le
\frac12\,\|F\|^2$ and
provides us with the norm equivalence of $\mathsf{H}[F]$ and
$\|F\|^2$, \be{H-norm}
\frac12\,(1-\delta)\,\|F\|^2\le\mathsf{H}[F]\le\frac12\,(1+\delta)\,\|F\|^2\,.
\ee With $X:=\|(1-\mathsf\Pi)F\|$ and $Y:=\|\mathsf\Pi F\|$, it
follows from~\eqref{A4} that
\[
\mathsf{D}[F]\ge(\lambda_m-\delta)\,X^2+\frac{\delta\,\lambda_M}{1+\lambda_M}\,Y^2-\delta\,C_M\,X\,Y\,.
\]
The choice $\delta=\frac12\,\min\left\{1,\lambda_m,\frac{\lambda_m\,\lambda_M}{(1+\lambda_M)\,C_M^2}\right\}$ implies that
\[
\mathsf{D}[F]\ge\frac{\lambda_m}4\,X^2+\frac{\delta\,\lambda_M}{2\,(1+\lambda_M)}\,Y^2\ge\frac14\,\min\left\{\lambda_m,\frac{2\,\delta\,\lambda_M}{1+\lambda_M}\right\}\|F\|^2\ge\frac{2\,\delta\,\lambda_M}{3\,(1+\lambda_M)}\,\mathsf{H}[F]\,.
\]
With $\lambda$ defined in~\eqref{decay-const}, using $\delta\le1/2$ and $(1+\delta)/(1-\delta)\le3$, we get
\[
\|F(t)\|^2\le\frac2{1-\delta}\,\mathsf{H}[F](t)\le\frac{1+\delta}{1-\delta}\,e^{-\lambda\,t}\,\|F(0)\|^2\le3\,e^{-\lambda\,t}\,\|F(0)\|^2\,.
\]
\end{proof}

For any fixed $\xi\in\R^d$, let us apply Proposition~\ref{theo:DMS} to~\eqref{eq:hat} with $F=\hat f$ and
\[
\mathcal H=\mathrm L^2\(d\gamma_\infty\)\,,\quad\|F\|^2=\int_{\R^d}|F|^2\,d\gamma_\infty\,,\quad
\mathsf\Pi F=M\int_{\R^d}F\,dv=M\,\rho_F\,,\quad\sfT F=i\,(v\cdot\xi)\,F \,.
\]
Here we are in a mode-by-mode framework in which the transport operator $\sfT$ is a simple multiplication operator.
%---------------------------------------------------------------------
\begin{corollary}\phantomsection\label{lec14-modelemm} Assume~\eqref{H1}--\eqref{H4}, and take $\xi\in\R^d$. If $\hat f$ is a solution of~\eqref{eq:hat} such that $\hat f_0(\xi,\cdot)\in\mathrm L^2(d\gamma_\infty)$, then for any $t\ge0$, we have
\[
\left\|\hat f(t,\xi,\cdot)\right\|_{\mathrm L^2\(d\gamma_\infty\)}^2
\le 3\,e^{-\,\mu_\xi\,t}\left\|\hat f_0(\xi,\cdot)\right\|_{\mathrm L^2\(d\gamma_\infty\)}^2 \,,
\]
where
\be{eq:mu_xi}
\mu_\xi:=\frac{\Lambda\,|\xi|^2}{1+|\xi|^2}\quad\mbox{and}\quad\Lambda=\frac13\,\min\big\{1,\Theta\big\}\,\min\left\{1,\frac{\lambda_m\,\Theta^2}{K+\Theta\,\kappa^2}\right\} \,,
\ee
with
\be{Theta}
\Theta:=\int_{\R^d}(v\cdot\mathsf e)^2\,M(v)\,dv \,,\quad K:=\int_{\R^d}(v\cdot\mathsf e)^4\,M(v)\,dv\,,\quad
\theta := \frac4d\int_{\R^d}\left|\nabla_v\sqrt M\right|^2\,dv \,,
\ee
for an arbitrary $\mathsf e\in\mathbb S^{d-1}$, and with $\kappa=\sqrt\theta$ in Case (a) and 
$\kappa=2\,\overline\sigma\,\sqrt\Theta$ in Case (b). \end{corollary}
%---------------------------------------------------------------------
\begin{proof} We check that the assumptions of
  Proposition~\ref{theo:DMS} are satisfied with \hbox{$F=\hat f$}. The
  property $\mathsf L^*\mathsf A=0$ is a consequence of the mass
  conservation $\int_{\R^d} \sfL f\,dv = 0$ because
  $\mathsf{\Pi A}=\mathsf A$. Assumption~\eqref{H4}
  implies~\eqref{A1}. Concerning the macroscopic
  coercivity~\eqref{A2}, since
\[
\mathsf{T\Pi}F=i\,(v\cdot\xi)\,\rho_F\,M\,,
\]
one has
\[
\|\mathsf{T\Pi}F\|^2 =|\rho_F|^2 \int_{\R^d}|v\cdot\xi|^2\,M(v)\,dv\;=\Theta\,|\xi|^2\,|\rho_F|^2=\Theta\,|\xi|^2\,\|\mathsf{\Pi}F\|^2\,,
\]
and thus~\eqref{A2} holds with $\lambda_M=\Theta\,|\xi|^2$. By
assumption $M(v)$ depends only on~$|v|$, so it is \emph{unbiased}:
$\int_{\R^d}v\,M(v)\,dv=0$, which means that~\eqref{A3} holds.
\smallskip

Let us now prove~\eqref{A4}. Since
$(\mathsf T\mathsf\Pi)^*F=-\,\mathsf\Pi\,\mathsf TF=-\,i\,(
\xi\cdot\int_{\R^d}v'\,F(v')\,dv')\,M$, we obtain that
\[
\big(1+(\mathsf T\mathsf\Pi)^*\mathsf T\mathsf\Pi\big)\,\rho\,M=\left(1+\int_{\R^d}\big(\xi\cdot v'\big)^2\,M(v')\,dv'\right)\,\rho\,M=\left(1+\Theta\,|\xi|^2\right)\,\rho\,M
\]
and the operator $\mathsf{A}$, defined in Proposition \ref{theo:DMS}, is given mode-by-mode by
\[
\mathsf{A}F=\frac{-\,i\,\xi\cdot\int_{\R^d}v'\,F(v')\,dv'}{1+\Theta\,|\xi|^2}\,M\,.
\]
As a consequence, $\mathsf{A}$ satisfies the estimate
\begin{multline*}
\|\mathsf{A}F\|=\|\mathsf{A}\mathsf{(1-\Pi)}F\|\le\frac1{1+\Theta\,|\xi|^2}\int_{\R^d}\frac{|\mathsf{(1-\Pi)}F|}{\sqrt M}\;|v\cdot\xi|\,\sqrt M\,dv\\
\le\frac{\|\mathsf{(1-\Pi)}F\|}{1+\Theta\,|\xi|^2}\(\int_{\R^d}(v\cdot\xi)^2\,M\,dv\)^{1/2}=\frac{\sqrt\Theta\,|\xi|}{1+\Theta\,|\xi|^2}\,\|\mathsf{(1-\Pi)}F\|\,.
\end{multline*}
In Case~(b) the collision operator $\mathsf{L}$ is obviously bounded:
\[
\|\mathsf{L}F\|\le 2\,\overline\sigma{\kern 0.5pt}\,\|\mathsf{(1-\Pi)}F\|
\]
and, as a consequence,
\[\
\|\mathsf{AL}F\|\le\frac{2\,\overline\sigma\,\sqrt\Theta\,|\xi|}{1+\Theta\,|\xi|^2}\,\|\mathsf{(1-\Pi)}F\|\,.
\]
We also notice that $\mathsf L^*\mathsf A=0$ according to~\eqref{H3}. For estimating $\mathsf{AL}$ in Case~(a), we note that
\[
\int_{\R^d} v\,\mathsf{L}F\,dv=2\int_{\R^d}\nabla_v\sqrt M\,\frac F{\sqrt M}\,dv
\]
and obtain as above that
\[
\|\mathsf{AL}F\|\le\frac2{1+\Theta\,|\xi|^2}\int_{\R^d}\frac{|\mathsf{(1-\Pi)}F|}{\sqrt M}\;\left|\xi\cdot\nabla_v\sqrt M\right|\,dv\le\frac{\sqrt\theta\,|\xi|}{1+\Theta\,|\xi|^2}\,\|\mathsf{(1-\Pi)}F\|\,.
\]
For both cases we finally obtain
\[\label{ALF}
\|\mathsf{AL}F\|\le\frac{\kappa\,|\xi|}{1+\Theta\,|\xi|^2}\,\|\mathsf{(1-\Pi)}F\|\,.
\]

Similarly we can estimate $\mathsf{AT(1-\Pi)}F=\frac{\int_{\R^d} \left( v' \cdot \xi \right)^2 (1-\Pi)F(v')\,dv'}{1+\Theta\,|\xi|^2}\,M$ by
\begin{align*}
\|\mathsf{AT(1-\Pi)}F\| & = \frac{\left| \int_{\R^d} \left( v' \cdot \xi \right)^2 (1-\Pi)F(v')\,dv' \right|}{1+\Theta\,|\xi|^2}\\
&\le \frac{\left(\int_{\R^d} \left( v' \cdot \xi \right)^4\,M (v')\,dv' \right)^{1/2} }{1+\Theta\,|\xi|^2}\,\|\mathsf{(1-\Pi)}F\|=\frac{\sqrt{K}\,|\xi|^2}{1+\Theta\,|\xi|^2}\,\|\mathsf{(1-\Pi)}F\|\,,
\end{align*}
meaning that we have proven~\eqref{A4} with $C_M=\frac{\kappa\,|\xi|+ \sqrt{K}\,|\xi|^2}{1+\Theta\,|\xi|^2}$.

With the elementary estimates
\[
  \frac{\Theta\,|\xi|^2}{1+\Theta\,|\xi|^2}\ge
  \min\big\{1,\Theta\big\}\,\frac{|\xi|^2}{1+|\xi|^2}
  \quad\mbox{and}\quad
  \frac{\lambda_M}{(1+\lambda_M)\,C_M^2}=
  \frac{\Theta\(1+\Theta\,|\xi|^2\)}{\(\kappa+\sqrt{K}\,|\xi|\)^2}
  \ge\frac{\Theta^2}{K+\Theta\,\kappa^2}\,,
\]
the proof is completed using~\eqref{decay-const}.\end{proof}

%%%%%%%%%%%%%%%%%%%%%%%%%%%%%%%%%%%%%%%%%%%%%%%%%%%%%%%%%%%%%%%%%%%%%%
%%%%%%%%%%%%%%%%%%%%%%%%%%%%%%%%%%%%%%%%%%%%%%%%%%%%%%%%%%%%%%%%%%%%%%
\section{Enlarging and shrinking spaces by factorization}\label{sec:factorization}

Square integrability against the inverse of the \emph{local equilibrium} $M$ is a rather restrictive assumption on the initial datum. In this section it will be relaxed with the help of the abstract \emph{factorization method} of~\cite{MR3779780} in a simple case (factorization of order $1$). Here we state the result and sketch a proof in a special case, for the convenience of the reader. We shall then give a result based on similar computations in the opposite direction: how to establish a rate in a stronger norm, which correspond to a \emph{shrinking} of the functional space. We will conclude with an application to the problem studied in Corollary~\ref{lec14-modelemm}. Let us start by \emph{enlarging} the space.
%---------------------------------------------------------------------
\begin{proposition}\phantomsection\label{th:factorization} Let $\mathcal B_1$, $\mathcal B_2$ be Banach spaces and let $\mathcal B_2$ be continuously imbedded in $\mathcal B_1$, \emph{i.e.},~$\|\cdot\|_1\le c_1\|\cdot\|_2$. Let $\mathfrak B$ and $\mathfrak A+\mathfrak B$ be the generators of the strongly continuous semigroups $e^{\mathfrak B\,t}$ and $e^{(\mathfrak A+\mathfrak B)\,t}$ on $\mathcal B_1$. Assume that there are positive constants $c_2$,
$c_3$, $c_4$, $\lambda_1$ and $\lambda_2$ such that, for all $t\ge 0$,
\[
\left\|e^{(\mathfrak A+\mathfrak B)\,t}\right\|_{2\to2}\le c_2\,e^{-\lambda_2\,t}\,,\quad\left\|e^{\mathfrak Bt}\right\|_{1\to1}\le c_3\,e^{-\lambda_1\,t}\,,\quad \left\|\mathfrak A\right\|_{1\to2}\le c_4\,,
\]
where $\|\cdot\|_{i\to j}$ denotes the operator norm for linear mappings from $\mathcal B_i$ to $\mathcal B_j$.
Then there exists a positive constant $C=C(c_1,c_2,c_3,c_4)$ such that, for all $t\ge 0$,
\[
\left\|e^{(\mathfrak A+\mathfrak B)\,t}\right\|_{1\to1}\le\left\{\begin{array}{ll}
C\(1+|\lambda_1-\lambda_2|^{-1}\)\,e^{-\min\{\lambda_1,\lambda_2\}\,t}&\quad\mbox{for}\;\lambda_1\ne\lambda_2\,,\\[4pt]
C\,(1+t)\,e^{-\lambda_1\,t}&\quad\mbox{for}\;\lambda_1=\lambda_2\,.\end{array}\right.
\]
\end{proposition}
%---------------------------------------------------------------------
\begin{proof} Integrating the identity $\frac d{ds}\(e^{(\mathfrak A+\mathfrak B)\,s}\,e^{\mathfrak B\,(t-s)}\)=e^{(\mathfrak A+\mathfrak B)\,s}\,\mathfrak A\,e^{\mathfrak B\,(t-s)}$ with respect to $s\in [0,t]$ gives
\[
e^{(\mathfrak A+\mathfrak B)\,t}=e^{\mathfrak B\,t}+\int_0^te^{(\mathfrak A+\mathfrak B)\,s}\,\mathfrak A\,e^{\mathfrak B\,(t-s)}\,ds\,.
\]
The proof is completed by the straightforward computation
\begin{multline*}
\big\|e^{(\mathfrak A+\mathfrak B)\,t}\big\|_{1\to1}\le c_3\,e^{-\lambda_1\,t}+c_1\int_0^t\big\|e^{(\mathfrak A+\mathfrak B)\,s}\,\mathfrak A\,e^{\mathfrak B\,(t-s)}\big\|_{1\to2}\,ds\\
\le c_3\,e^{-\lambda_1\,t}+c_1\,c_2\,c_3\,c_4\,e^{-\lambda_1\,t}\int_0^te^{(\lambda_1-\lambda_2)\,s}\,ds\,.
\end{multline*}
\end{proof}

The second statement of this section is devoted to a result on the
\emph{shrinking} of the functional space. It is based on a computation
which is similar to the one of the proof of
Proposition~\ref{th:factorization}.
%---------------------------------------------------------------------
\begin{proposition}\phantomsection\label{prop:shrink} Let
  $\mathcal B_1$, $\mathcal B_2$ be Banach spaces and let
  $\mathcal B_2$ be continuously imbedded in $\mathcal B_1$,
  \emph{i.e.},~$\|\cdot\|_1\le c_1\|\cdot\|_2$. Let $\mathfrak B$ and
  $\mathfrak A+\mathfrak B$ be the generators of the strongly
  continuous semigroups $e^{\mathfrak B\,t}$ and
  $e^{(\mathfrak A+\mathfrak B)\,t}$ on $\mathcal B_1$. Assume that
  there are positive constants $c_2$,
$c_3$, $c_4$, $\lambda_1$ and $\lambda_2$ such that, for all $t\ge 0$,
\[
\left\|e^{(\mathfrak A+\mathfrak B)\,t}\right\|_{1\to1}\le c_2\,e^{-\lambda_1\,t}\,,\quad\left\|e^{\mathfrak Bt}\right\|_{2\to2}\le c_3\,e^{-\lambda_2\,t}\,,\quad \left\|\mathfrak A\right\|_{1\to2}\le c_4\,,
\]
where $\|\cdot\|_{i\to j}$ denotes the operator norm for linear mappings from $\mathcal B_i$ to $\mathcal B_j$.
Then there exists a positive constant $C=C(c_1,c_2,c_3,c_4)$ such that, for all $t\ge 0$,
\[
\left\|e^{(\mathfrak A+\mathfrak B)\,t}\right\|_{2\to2}\le\left\{\begin{array}{ll}
C\(1+|\lambda_2-\lambda_1|^{-1}\)\,e^{-\min\{\lambda_2,\lambda_1\}\,t}&\quad\mbox{for}\;\lambda_2\ne\lambda_1\,,\\[4pt]
C\,(1+t)\,e^{-\lambda_1\,t}&\quad\mbox{for}\;\lambda_1=\lambda_2\,.\end{array}\right.
\]
\end{proposition}
%---------------------------------------------------------------------
\begin{proof} Integrating the identity $\frac d{ds}\(e^{\mathfrak B\,(t-s)}e^{(\mathfrak A+\mathfrak B)\,s}\,\)=e^{\mathfrak B\,(t-s)}\,\mathfrak A\,e^{(\mathfrak A+\mathfrak B)\,s}$ with respect to $s\in [0,t]$ gives
\[
e^{(\mathfrak A+\mathfrak B)\,t}=e^{\mathfrak B\,t}+\int_0^t e^{\mathfrak B\,(t-s)}\,\mathfrak A\,e^{(\mathfrak A+\mathfrak B)\,s}\,ds\,.
\]
The proof is completed by the straightforward computation
\begin{align*}
\big\|e^{(\mathfrak A+\mathfrak B)\,t}\big\|_{2\to2}&\le c_3\,e^{-\lambda_2\,t}+\int_0^t\big\|e^{\mathfrak B\,(t-s)}\,\mathfrak A\,e^{(\mathfrak A+\mathfrak B)\,s}\big\|_{2\to2}\,ds\\
&\le c_3\,e^{-\lambda_2\,t}+c_1\int_0^t\big\|e^{\mathfrak B\,(t-s)}\,\mathfrak A\,e^{(\mathfrak A+\mathfrak B)\,s}\big\|_{1\to2}\,ds\\
&\le c_3\,e^{-\lambda_2\,t}+c_1\int_0^t\big\|e^{\mathfrak B\,(t-s)}\big\|_{2\to2}\,\big\|\mathfrak A\big\|_{1\to2}\,\big\|e^{(\mathfrak A+\mathfrak B)\,s}\big\|_{1\to1}\,ds\\
&\le c_3\,e^{-\lambda_2\,t}+c_1\,c_2\,c_3\,c_4\,e^{-\lambda_2\,t}\int_0^te^{(\lambda_2-\lambda_1)\,s}\,ds\,.
\end{align*}
\end{proof}

We will use Proposition~\ref{prop:shrink} in Section~\ref{Sec:ImprovedDecayRates}. Coming back to the problem studied in Corollary~\ref{lec14-modelemm}, Proposition~\ref{th:factorization} applies to~\eqref{eq:hat} with the spaces $\mathcal B_1=\mathrm L^2(d\gamma_k)$, $k\in (d,\infty)$, and $\mathcal B_2=\mathrm L^2(d\gamma_\infty)$ corresponding to the weights defined by~\eqref{Def:Measures}. The exponential growth of ~$\gamma_\infty$ guarantees that $\mathcal B_2$ is continuously imbedded in $\mathcal B_1$.
%---------------------------------------------------------------------
\begin{corollary}\phantomsection\label{th:fac-appl} Assume~\eqref{H1}--\eqref{H4}, $k\in(d,\infty]$, and $\xi\in\R^d$. Then there exists a constant $C>0$,
such that solutions $\hat f$ of~\eqref{eq:hat} with initial datum $\hat f_0(\xi,\cdot)\in\mathrm L^2(d\gamma_k)$ satisfy, 
with $\mu_\xi$ given by~\eqref{eq:mu_xi},
\[
\left\|\hat f(t,\xi,\cdot)\right\|_{\mathrm L^2(d\gamma_k)}^2\le C\,e^{-\,\mu_\xi\,t}\,\left\|\hat f_0(\xi,\cdot)\right\|_{\mathrm L^2(d\gamma_k)}^2\quad\forall\,t\ge 0\,.
\]
\end{corollary}
%---------------------------------------------------------------------
\begin{proof} In Case~(a), let us define $\mathfrak A$ and $\mathfrak B$ by $\mathfrak AF=N\,\chi_RF$ and $\mathfrak BF=-\,i\,(v\cdot\xi)\,F+\mathsf LF-\mathfrak AF$, where $N$ and $R$ are two positive constants, $\chi$ is a smooth function such that $\mathbb 1_{B_1}\le\chi\le\mathbb{1}_{B_2}$, and $\chi_R:=\chi(\cdot/R)$. Here $B_r$ is the centered ball of radius $r$. It has been established in~\cite[Lemma~3.8]{MisMou} that if $k>d$, then the inequality 
\[
\int_{\R^d}(\sfL-\mathsf{\mathfrak A})(F)\,F\,d\gamma_k\le-\,\lambda_1\int_{\R^d}F^2\,d\gamma_k
\]
holds for some $\lambda_1>0$. Moreover, $\lambda_1$ can be chosen arbitrarily large for $R$ and $N$ large enough. The boundedness of $\mathfrak A:\,\mathcal B_1\to\mathcal B_2$ follows from the compactness of the support of~$\chi$ and Proposition~\ref{th:factorization} applies with $\lambda_2=\mu_\xi/2\le1/4$, where $\mu_\xi$ is given by~\eqref{eq:mu_xi}.

In Case~(b), we consider $\mathfrak A$ and $\mathfrak B$ such that
\begin{eqnarray*}
\mathfrak AF(v)&=&M(v)\int_{\R^d}\sigma(v,v')\,F(v')\,dv'\,,\\
\mathfrak BF(v)&=&-\left[i\,(v\cdot\xi)\,+\int_{\R^d}\sigma(v,v')\,M(v')\,dv'\right]\,F(v)\,.
\end{eqnarray*}
The boundedness of $\mathfrak A:\,\mathcal B_1\to\mathcal B_2$ follows from~\eqref{H2} and
\[\label{A-bound2}
\|\mathfrak AF\|_{\mathrm L^2(d\gamma_\infty)}\le\overline\sigma\,\|F\|_{\mathrm L^1(dv)}\le\overline\sigma\,\(\int_{\R^d}\gamma_k^{-1}\,dv\)^{1/2}\|F\|_{\mathrm L^2(d\gamma_k)}\,.
\]
Proposition~\ref{th:factorization} applies with $\lambda_2=\frac{\mu_\xi}2\le\frac14$ and $\lambda_1=1$ because \hbox{$\int_{\R^d}\sigma(v,v')\,M(v')\,dv'\ge1$}. \end{proof}

%%%%%%%%%%%%%%%%%%%%%%%%%%%%%%%%%%%%%%%%%%%%%%%%%%%%%%%%%%%%%%%%%%%%%%
%%%%%%%%%%%%%%%%%%%%%%%%%%%%%%%%%%%%%%%%%%%%%%%%%%%%%%%%%%%%%%%%%%%%%%
\section{Asymptotic behavior based on mode-by-mode estimates}\label{Sec:AsymptoticFourier}

In this section we consider~\eqref{eq:model} and use the estimates of Corollary~\ref{lec14-modelemm} with weight $\gamma_\infty=1/M$ and Corollary~\ref{th:fac-appl} for weights with $O(|v|^k)$ growth to get decay rates with respect to $t$. We shall consider two cases for the spatial variable $x$. In Section~\ref{sec:recov-spectr-gap}, we assume that $x\in\T^d$, where $\T^d$ is the flat $d$-dimensional torus (represented by $[0,2\pi)^d$ with periodic boundary conditions) and prove an exponential convergence rate. In Section~\ref{Sec:WholeSpaceFourier}, we assume that $x\in\R^d$ and establish algebraic decay rates.

%%%%%%%%%%%%%%%%%%%%%%%%%%%%%%%%%%%%%%%%%%%%%%%%%%%%%%%%%%%%%%%%%%%%%%
\subsection{Exponential convergence to equilibrium in
  \texorpdfstring{$\T^d$}{Td}}
\label{sec:recov-spectr-gap}
% As a first example of our mode-by-mode estimates, we consider the case
% $x\in\T^d$.
In the periodic case $x\in\T^d$ there is a unique non-zero normalized
equilibrium given by
\[\label{f_infty}
  f_\infty(x,v)=\rho_\infty\,M(v)
  \quad \text{with} \quad
  \rho_\infty=\frac1{|\T^d|}\iint_{\T^d\times\R^d}f_0\,dx\,dv\,.
\]
%---------------------------------------------------------------------
\begin{corollary}\phantomsection\label{Thm:torus}
  Assume~\eqref{H1}--\eqref{H4} and $k\in(d,\infty]$. Then there
  exists a constant $C>0$, such that the solution $f$
  of~\eqref{eq:model} on $\T^d\times\R^d$ with initial datum
  $f_0\in\mathrm L^2(dx\,d\gamma_k)$ satisfies, with $\Lambda$ given
  by~\eqref{eq:mu_xi},
  \[
    \left\|f(t,\cdot,\cdot)-f_\infty\right\|_{\mathrm
      L^2(dx\,d\gamma_k)}\le C\,\left\|f_0-f_\infty\right\|_{\mathrm
      L^2(dx\,d\gamma_k)}\, e^{-\Lambda
      \frac{t}{4}}\quad\forall\,t\ge0\,.
  \]
\end{corollary}
%---------------------------------------------------------------------
\begin{proof}
  We represent the flat torus $\T^d$ by $[0,2\pi)^d$ with periodic
  boundary conditions, and the Fourier variable is denoted
  $\xi\in\Z^d$. For $\xi=0$, the microscopic coercivity (see Section
  \ref{sec:mode-mode-hypoc}) implies
  \[
    \left\|\hat f(t,0,\cdot)-\hat f_\infty(0,\cdot)\right\|_{\mathrm
      L^2(d\gamma_\infty)}\le \left\|\hat f_0(0,\cdot)-\hat
      f_\infty(0,\cdot)\,\right\|_{\mathrm
      L^2(d\gamma_\infty)}\,e^{-t}\,.
  \]
  For all other modes, $\hat f_\infty(\xi,\cdot)=0$ for any $\xi\ne 0$
  (that is, for any $\xi$ such that $|\xi|\ge1$). We can use
  Corollary~\ref{lec14-modelemm} with $\mu_\xi\ge\Lambda/2$, with the
  notations of~\eqref{eq:mu_xi}. An application of Parseval's identity
  then proves the result for $k=\infty$, and $C=\sqrt3$.  If~$k$ is
  finite, the result with the weight $\gamma_k$ follows from
  Corollary~\ref{th:fac-appl}.
\end{proof}

Note that the latter result can also alternatively be proved by
directly applying Proposition~\ref{theo:DMS} to~\eqref{eq:model}, as
in~\cite{DMS-2part}.

%%%%%%%%%%%%%%%%%%%%%%%%%%%%%%%%%%%%%%%%%%%%%%%%%%%%%%%%%%%%%%%%%%%%%%
\subsection{Algebraic decay rates in \texorpdfstring{$\R^d$}{Rd}}\label{Sec:WholeSpaceFourier}

With the result of Corollary~\ref{lec14-modelemm} and Corollary~\ref{th:fac-appl} we obtain a first proof of Theorem~\ref{th:whole-space-mode} as follows. Let $C>0$ be a generic constant which is going to change from line to line. Plancherel's formula implies
\[
\left\|f(t,\cdot,\cdot)\right\|_{\mathrm L^2(dx\,d\gamma_k)}^2\le C\int_{\R^d}\(\int_{\R^d}e^{-\,\mu_\xi\,t}\,\left|\hat f_0\right|^2\,d\xi\)\,d\gamma_k\,.
\]
We know that $\int_{|\xi|\le1}e^{-\,\mu_\xi\,t}\,d\xi\le\int_{\R^d}e^{-\frac\Lambda2\,|\xi|^2\,t}\,d\xi=\big(\frac{2\,\pi}{\Lambda\,t}\big)^{d/2}$ and thus, for all $v \in \R^d$,
\[
\int_{|\xi|\le1}\kern-5pt e^{-\,\mu_\xi\,t}\,\left|\hat f_0 \right|^2\,d\xi\le C\,\left\|f_0(\cdot,v)\right\|_{\mathrm L^1(dx)}^2\int_{\R^d}\kern-5pt e^{-\,\frac\Lambda2\,|\xi|^2\,t}\,d\xi\le C\,\left\|f_0(\cdot,v)\right\|_{\mathrm L^1(dx)}^2\,t^{-\frac d2}.
\]
Using the fact that $\mu_\xi\ge\Lambda/2$ when $|\xi|\ge1$ and Plancherel's formula, we know that, for all $v \in \R^d$,
\[
\int_{|\xi|>1}e^{-\,\mu_\xi\,t}\,\left|\hat f_0\right|^2\,d\xi\le C\,e^{-\frac\Lambda2\,t}\,\left\|f_0(\cdot,v)\right\|_{\mathrm L^2(dx)}^2\,,
\]
which completes a first proof of Theorem~\ref{th:whole-space-mode}.

%%%%%%%%%%%%%%%%%%%%%%%%%%%%%%%%%%%%%%%%%%%%%%%%%%%%%%%%%%%%%%%%%%%%%%
%%%%%%%%%%%%%%%%%%%%%%%%%%%%%%%%%%%%%%%%%%%%%%%%%%%%%%%%%%%%%%%%%%%%%%
\section{Hypocoercivity and the Nash inequality}\label{Sec:WholeSpaceNash}

In view of the proof of Theorem~\ref{th:whole-space-mode} in Section~\ref{Sec:WholeSpaceFourier} and of the rate, it is natural to wonder if the hypocoercivity can be controlled by the use of Nash's inequality. Here we temporarily abandon the Fourier variable~$\xi$ and consider the direct variable $x\in\R^d$: throughout this section, the \emph{transport operator} on the position space is defined as
\[
\mathsf T f=v\cdot\nabla_xf\,.
\]
We rely on the abstract setting of Section~\ref{sec:mode-mode-hypoc}, applied to~\eqref{eq:model} with the scalar product $\langle\cdot,\cdot\rangle$ on $\mathrm L^2(dx\,d\gamma_\infty)$ and the induced norm $\|\cdot\|$. Notice that this norm includes the $x$ variable, which was not the case in the mode-by-mode analysis of Section~\ref{sec:mode-mode-hypoc}. It is then easy to check that $(\mathsf T\mathsf\Pi)f=M\mathsf T\rho_f=v\cdot\nabla_x\rho_f\,M$, $(\mathsf T\mathsf\Pi)^*f=-\,\nabla_x\cdot\left(\int_{\R^d}v\,f\,dv\right)\,M$ and $(\mathsf T\mathsf\Pi)^*(\mathsf T\mathsf\Pi)f=-\,\Theta\,\Delta_x\rho_f\,M$ so that
\[
g=\mathsf{A}f=\big(1+(\mathsf T\mathsf\Pi)^*\mathsf T\mathsf\Pi\big)^{-1}(\mathsf T\mathsf\Pi)^*f\quad\Longleftrightarrow\quad g=u\,M
\]
where $u-\,\Theta\,\Delta u =-\,\nabla_x\cdot \left( \int_{\R^d}v\,f\,dv\right)$. Since $M$ is unbiased, $\mathsf{A}f=\mathsf{A(1-\Pi)}f$. For some $\delta>0$ to be chosen later, we redefine the entropy~by $\mathsf{H}[f]:=\tfrac12\,\|f\|^2+\delta\,\langle\mathsf Af,f\rangle$.

\begin{proof}[Proof of Theorem \ref{th:whole-space-mode}] If $f$ solves~\eqref{eq:model}, the time derivative of $\mathsf{H}[f(t,\cdot,\cdot)]$ is given by
\be{entr-dec}
\frac d{dt}\mathsf{H}[f]=-\,\mathsf{D}[f]
\ee
where, as in the proof of Proposition~\ref{theo:DMS}, 
\[
\mathsf{D}[f]:=-\,\langle\mathsf{L}f,f\rangle+\delta\,\langle\mathsf{AT}\mathsf\Pi f,f\rangle+\delta\,\re\langle\mathsf{AT}(1-\mathsf\Pi)f,f\rangle-\delta\,\re\langle\mathsf{TA}f,f\rangle-\,\delta\,\re\langle\mathsf{AL}f,f\rangle\,.
\]
Here we use the fact that $\langle\mathsf{A}f,\mathsf{L}f\rangle = 0$. The first term in $\mathsf{D}[f]$ satisfies the microscopic coercivity condition
\[
-\,\langle\mathsf{L}f,f\rangle\ge\lambda_m\,\|\mathsf{(1-\Pi)}f\|^2\,.
\]
The second term in~\eqref{entr-dec} is computed as follows. Solving $g=\mathsf{AT\Pi}f$ is equivalent to solving $\(1+(\mathsf{T\Pi})^*\mathsf{T\Pi}\)g=(\mathsf{T\Pi})^*\mathsf{T\Pi}f$, \emph{i.e.},
\be{Eqn:vf}
v_f-\,\Theta\,\Delta_x\,v_f=-\,\Theta\,\Delta_x\rho_f\,,
\ee
where $g=v_f\,M$. Hence
\[
\langle\mathsf{AT\Pi}f,f\rangle=\int_{\R^d}v_f\,\rho_f\,dx\,.
\]

A direct application of the hypocoercivity approach of~\cite{DMS-2part} to the whole space problem fails by lack of a \emph{macroscopic coercivity} condition. Although the second term in~\eqref{entr-dec} is not coercive, we observe that the last three terms in~\eqref{entr-dec} can still be dominated by the first two for $\delta>0$, small enough, as follows.
\\
1) As in~\cite{DMS-2part}, we use the adjoint operators to compute
\[
\langle\mathsf{AT}\mathsf{(1-\Pi)}f,f\rangle=-\,\langle\mathsf{(1-\Pi)}f,\mathsf{TA}^*f\rangle\,.
\]
We observe that
\[
\mathsf A^*f=\mathsf{T\Pi}\,\big(1+(\mathsf{T\Pi})^*\mathsf{T\Pi}\big)^{-1}f=\mathsf T\,\big(1+(\mathsf{T\Pi})^*\mathsf{T\Pi}\big)^{-1}\,\mathsf{\Pi}f=M\,\mathsf Tu_f=v\,M\cdot\nabla_xu_f
\]
where $u_f$ is the solution in $\mathrm H^1(dx)$ of
\be{Eqn:uf}
u_f-\,\Theta\,\Delta_xu_f=\rho_f\,.
\ee
With $K$ defined by~\eqref{Theta}, we obtain that
\[
\left\|\mathsf{TA}^*f \right\|^2\le K\,\left\|\nabla_x^2u_f\right\|_{\mathrm L^2(dx)}^2=K\,\left\|\Delta_xu_f\right\|_{\mathrm L^2(dx)}^2\,.
\]
On the other hand, we observe that $v_f=-\,\Theta\,\Delta u_f$
solves~\eqref{Eqn:vf}. Hence by multiplying~\eqref{Eqn:uf} by
$v_f=-\,\Theta\,\Delta u_f$ and integrating by parts, we know that
\be{ATPi} \Theta\,\left\|\nabla_xu_f\right\|_{\mathrm
  L^2(dx)}^2+\,\Theta^2\,\left\|\Delta_xu_f \right\|_{\mathrm
  L^2(dx)}^2
=\int_{\R^d}v_f\,\rho_f\,dx=\langle\mathsf{AT\Pi}f,f\rangle\,.  \ee
Notice that a central feature of our method is the fact that
quantities of interest involving the operator $\mathsf A$ can be
computed by solving an elliptic equation (for instance~\eqref{Eqn:vf}
in case of $\mathsf{AT\Pi}f$ or~\eqref{Eqn:uf} in case of
$\mathsf A^*f$). Altogether we obtain that
\[
|\langle\mathsf{AT}\mathsf{(1-\Pi)}f,f\rangle|\le\|\mathsf{(1-\Pi)}f\|\,\|\mathsf{TA}^*f\| 
\le \frac{\sqrt{K}}{\Theta}\,\|\mathsf{(1-\Pi)}f\|\,\langle\mathsf{AT\Pi}f,f\rangle^{1/2}\,.
\]
2) By~\eqref{A-bound}, we have
\[\label{TA}
\left|\langle\mathsf{TA}f,f\rangle\right|=\left|\langle\mathsf{TA}\mathsf{(1-\Pi)}f,\mathsf{(1-\Pi)}f\rangle\right|\le\|\mathsf{(1-\Pi)}f\|^2\,.
\]
3) It remains to estimate the last term on the right hand side of~\eqref{entr-dec}. Let us consider the solution $u_f$ of~\eqref{Eqn:uf}. If we multiply~\eqref{Eqn:vf} by $u_f$ and integrate, we observe that
\[
\Theta \left\|\nabla_x u_f\right\|_{\mathrm L^2(dx)}^2 =\int_{\R^d}u_f\,v_f\,dx\le\int_{\R^d}u_f\,v_f\,dx+\int_{\R^d}|v_f|^2\,dx=\int_{\R^d}v_f\,\rho_f\,dx
\]
because $v_f=-\,\Theta\,\Delta u_f$, so that
\[
\|\mathsf A^*f\|^2= \Theta \left\|\nabla_x u_f\right\|_{\mathrm L^2(dx)}^2 \le\langle\mathsf{AT\Pi}f,f\rangle\,.
\]
In Case~(a), we compute
\[
\langle\mathsf{AL}f,f\rangle=\langle\mathsf{L(1-\Pi)}f,\mathsf{A}^*f\rangle=\iint_{\R^d\times\R^d}\nabla_xu_f\cdot\frac{\nabla_vM}M\,\mathsf{(1-\Pi)}f\,dx\,dv\,.
\]
It follows from the Cauchy-Schwarz inequality that
\begin{multline*}
\int_{\R^d}|\nabla_vM|\,|\mathsf{(1-\Pi)}f|\,d\gamma_\infty 
\le\|\nabla_vM\|_{\mathrm L^2\left(d\gamma_\infty\right)}\|\mathsf{(1-\Pi)}f\|_{\mathrm L^2\left(d\gamma_\infty\right)}\\
= \sqrt{d\,\theta}\,\|\mathsf{(1-\Pi)}f\|_{\mathrm L^2\left(d\gamma_\infty\right)}
\end{multline*}
and
\[
|\langle\mathsf{AL}f,f\rangle|\le\left\|\nabla_xu_f\right\|_{\mathrm L^2(dx)}\left(\int_{\R^d}\left(\frac1d\int_{\R^d}|\nabla_vM|\,|\mathsf{(1-\Pi)}f|\,d\gamma\right)^2\,dx\right)^\frac12\,.
\]
Altogether, we obtain that
\[
|\langle\mathsf{AL}f,f\rangle|\le\sqrt{\frac\theta\Theta}\,\|\mathsf{(1-\Pi)}f\|\,\langle\mathsf{AT\Pi}f,f\rangle^\frac12\,.
\]
In Case (b), we use~\eqref{H2} to get that
\[
|\langle\mathsf{AL}f,f\rangle|\le\|\mathsf{L}f\|\,\|\mathsf{A}^*f\|\le2\,\overline\sigma\,\|\mathsf{(1-\Pi)}f\|\,\|\mathsf A^*f\|\le2\,\overline\sigma\,\|\mathsf{(1-\Pi)}f\|\,\langle\mathsf{AT\Pi}f,f\rangle^\frac12\,.
\]
In both cases, (a) and (b), the estimate can be written as
\[
|\langle\mathsf{AL}f,f\rangle|\le2\,\overline\sigma\,\|\mathsf{(1-\Pi)}f\|\,\langle\mathsf{AT\Pi}f,f\rangle^\frac12
\]
with the convention that $\overline\sigma=\frac12\,\sqrt{\theta/\Theta}$ in Case (a).

Summarizing, we know that
\[
-\frac d{dt}\mathsf{H}[f]\ge(\lambda_m-\delta)\,X^2+\delta\,Y^2+2\,\delta\,\mathsf b\,X\,Y
\]
with $X:=\|\mathsf{(1-\Pi)}f\|$, $Y:=\langle\mathsf{AT\Pi}f,f\rangle^{1/2}$ and $\mathsf b:=\tfrac K{2\,\Theta}+2\,\overline\sigma$. The largest $\mathsf a>0$ such that
\[
(\lambda_m-\delta)\,X^2+\delta\,Y^2+2\,\delta\,\mathsf b\,X\,Y\ge\mathsf a\(X^2+2\,Y^2\)
\]
holds for any $X$, $Y\in\R$ is given by the conditions
\be{Conditions}
\mathsf a<\lambda_m-\delta\,,\quad2\,\mathsf a<\delta\,,\quad\delta^2\,\mathsf b^2-(\lambda_m-\delta-\mathsf a)\,(\delta-2\,\mathsf a)\le0
\ee
and it is easy to check that there exists a positive solution if $\delta>0$ is small enough. To fulfill the additional constraint $\delta<1$, we can for instance choose
\[
\delta=\frac{4\,\min\left\{1,\lambda_m\right\}}{8\,\mathsf b^2+5}\quad\mbox{and}\quad \mathsf a=\frac\delta{4}\,.
\]
Altogether we obtain that
\[
-\frac d{dt}\mathsf{H}[f]\ge\mathsf a\,\Big(\|\mathsf{(1-\Pi)}f\|^2+2\,\langle\mathsf{AT\Pi}f,f\rangle\Big)\,.
\]
Using~\eqref{Eqn:uf} and~\eqref{ATPi}, we control~$\|\mathsf{\Pi}f\|^2=\|\rho_f\|_{\mathrm L^2(dx)}^2$ by $\langle\mathsf{AT\Pi}f,f\rangle$ according~to
\begin{multline*}
\left\|\mathsf{\Pi}f\right\|^2=\left\|u_f\right\|_{\mathrm L^2(dx)}^2+2\,\Theta\,\left\|\nabla_xu_f \right\|_{\mathrm L^2(dx)}^2+\Theta^2\,\left\|\Delta_xu_f\right\|_{\mathrm L^2(dx)}^2\\
\le\left\|u_f\right\|_{\mathrm L^2(dx)}^2+2\,\langle\mathsf{AT\Pi}f,f\rangle\,.
\end{multline*}
We observe that, for any $t\ge0$,
\[
\left\|u_f(t,\cdot)\right\|_{\mathrm L^1(dx)}=\left\|\rho_f(t,\cdot)\right\|_{\mathrm L^1(dx)}=\left\|f_0\right\|_{\mathrm L^1(dx\,dv)}\,,\quad \left\|\nabla_xu_f\right\|_{\mathrm L^2(dx)}^2\le\frac1\Theta\,\langle\mathsf{AT\Pi}f,f\rangle\,.
\]
According to~\cite{Nash}, we recall the \emph{Nash inequality}
\be{Nash}
\|u\|_{\mathrm L^2(dx)}^2\le \mathcal C_\text{Nash}\,\|u\|_{\mathrm L^1(dx)}^\frac4{d+2}\,\|\nabla u\|_{\mathrm L^2(dx)}^\frac{2\,d}{d+2}
\ee
for any function $u\in\mathrm L^1\cap\mathrm H^1(\R^d)$. We use~\eqref{Nash} with $u=u_f$ to get
\[
\|\mathsf{\Pi}f\|^2\le\Phi^{-1}\bigl(2\,\langle\mathsf{AT\Pi}f,f\rangle\bigr)\quad\mbox{with}\quad\Phi^{-1}(y) := y+\(\frac y{\mathsf c}\)^\frac d{d+2}\quad\forall\,y\ge0
\]
where $\mathsf c=2\,\Theta\,\mathcal C_\text{Nash}^{-1-\frac 2d}\,\|f_0\|_{\mathrm L^1(dx\,dv)}^{-\frac4d}$. The function $\Phi:\,[0,\infty) \to [0,\infty)$ satisfies $\Phi(0)=0$ and $0<\Phi'<1$, so that
\[
\|\mathsf{(1-\Pi)}f\|^2+2\,\langle\mathsf{AT\Pi}f,f\rangle\ge\Phi(\|f\|^2)\ge\Phi\big(\tfrac2{1+\delta}\,\mathsf{H}[f]\big)
\]
where the last inequality holds as a consequence of~\eqref{H-norm}. From
\[
z = \Phi^{-1}(y) = y+\(\frac y{\mathsf c}\)^\frac d{d+2} \le y_0^{\frac 2{d+2}} y^{\frac d{d+2}} + \(\frac y{\mathsf c}\)^\frac d{d+2} = \left(y_0^{\frac 2{d+2}} + \mathsf c^{-\frac d{d+2}} \right)y^{\frac d{d+2}}\,,
\]
as long as $y \le y_0$, for $y_0$ to be chosen later, we have
\[
y=\Phi(z) \ge \left(\Phi(z_0)^{\frac 2{d+2}} + \mathsf c^{-\frac d{d+2}} \right)^{-\frac{d+2}d}\,z^{1+\frac2d},
\]
as long as $z \le z_0 := \Phi^{-1}(y_0)$. Since $\frac d{dt}\mathsf{H}[f] \le 0$, we have $\tfrac2{1+\delta}\,\mathsf{H}[f] \le \tfrac2{1+\delta}\,\mathsf{H}[f_0]$. We thus apply the previous inequalities with $z_0 = \tfrac2{1+\delta}\,\mathsf{H}[f_0]$ together with the fact that $\Phi(z_0) \ge z_0 \ge \tfrac{1-\delta}{1+\delta}\,\| f_0\|^2$ and that $\mathsf c$ is proportional to $\|f_0\|_{\mathrm L^1(dx\,dv)}^{-4/d}$, to get
\[
\Phi\big(\tfrac2{1+\delta}\,\mathsf{H}[f]\big) \gtrsim \left(\| f_0\|_{\mathrm L^2\left(dx\,d\gamma_\infty \right)}^{\frac 4{d+2}} +\|f_0\|_{\mathrm L^1(dx\,dv)}^\frac{4}{d+2} \right)^{-\frac{d+2}d}\,\mathsf{H}[f] ^{1+\frac2d}\,.
\]
We deduce the entropy decay inequality
\be{eq:finalestimate}
-\frac d{dt}\mathsf{H}[f] \gtrsim \left(\| f_0\|_{\mathrm L^2\left(dx\,d\gamma_\infty \right)}^{\frac 4{d+2}} +\|f_0\|_{\mathrm L^1(dx\,dv)}^\frac{4}{d+2} \right)^{-\frac{d+2}d}\,\mathsf{H}[f]^{1+\frac2d} .
\ee
A simple integration from $0$ to $t$ shows that
\[
\mathsf{H}[f]\lesssim\Big[\mathsf{H}[f_0]^{-\frac2d}+\Big(\| f_0\|_{\mathrm L^2\left(dx\,d\gamma_\infty \right)}^{\frac 4{d+2}} +\|f_0\|_{\mathrm L^1(dx\,dv)}^\frac{4}{d+2} \Big)^{-\frac{d+2}d}\,t\,\Big]^{-\frac d2}\,.
\]
The result of Theorem~\ref{th:whole-space-mode} then follows from elementary considerations.\end{proof}

Using moments instead of the mass, it is possible to state an \emph{improved Nash inequality}: there exists a positive constant $\mathcal C_\star$ such that
\[
\|u\|_{\mathrm L^2(dx)}^2 \le\mathcal C_\star \left\| x\,u \right\|_{\mathrm L^1(dx)}^\frac4{d+4}\,\|\nabla u\|_{\mathrm L^2(dx)}^\frac{d+2}{d+4}
\]
for any $u\in\mathrm H^1(dx)\cap\mathrm L^1\((1+|x|)\,dx\)$ such that $\int_{\R^d}u\,dx=0$. The proof follows from a minor modification of Nash's original proof (attributed by Nash himself to Stein) in~\cite{Nash} and uses Fourier variables. As a consequence, any solution of the heat equation with zero average decays in $\mathrm L^2(dx)$ like $O\big(t^{-1-d/2}\big)$ as $t\to+\infty$. It is the topic of the following section to use Fourier variables in the spirit of Nash's proof to get improved rates of decay at the level of the kinetic equation.

%%%%%%%%%%%%%%%%%%%%%%%%%%%%%%%%%%%%%%%%%%%%%%%%%%%%%%%%%%%%%%%%%%%%%%
%%%%%%%%%%%%%%%%%%%%%%%%%%%%%%%%%%%%%%%%%%%%%%%%%%%%%%%%%%%%%%%%%%%%%%
\section{Algebraic decay rates in \texorpdfstring{$\R^d$}{Rd} by Fourier estimates and improvements}\label{Sec:Fourier}

We prove Theorem~\ref{th:whole-space-zero-average} in Section~\ref{sec:impr-decay-rate} and  Theorem~\ref{th:whole-space-zero-average2} in Section~\ref{Sec:ImprovedDecayRates}.

%%%%%%%%%%%%%%%%%%%%%%%%%%%%%%%%%%%%%%%%%%%%%%%%%%%%%%%%%%%%%%%%%%%%%%
\subsection{Improved decay rates}\label{sec:impr-decay-rate}
Let us prove Theorem~\ref{th:whole-space-zero-average} by Fourier methods inspired by the proof of Nash's inequality.

%%%%%%%%%%%%%%%%%%%%%%%%%%%%%%%%%%%%%%%%%%%%%%%%%%%%%%%%%%%%%%%%%%%%%%
\subsubsection*{$\bullet$ Step 1: Decay of the average in space by a factorization argument}
We define
\be{eq:fav}
\fav(t,v):=\int_{\R^d}f(t,x,v)\,dx
\ee
and observe that $\fav$ solves
\[
\partial_t\fav=\mathsf L\fav\,.
\]
As a consequence, we have that $0=\int_{\R^d}\fav(t,v)\,dv$. From the \emph{microscopic coercivity property}~\eqref{H4}, we deduce that
\[
\left\|\fav(t,\cdot)\right\|_{\mathrm L^2(d\gamma_\infty)}^2=\int_{\R^d}\left|\tfrac{\fav(t,v)}M\right|^2\,M\,dv\le\left\|\fav(0,\cdot)\right\|_{\mathrm L^2(d\gamma_\infty)}^2\,e^{-\lambda_m\,t}\quad\forall\,t\ge0\,.
\]
With $k\in(d,\infty)$, Proposition~\ref{th:factorization} applies like in the proof of Corollary~\ref{th:fac-appl} or in~\cite{MisMou}. We observe that $\left\|\fav(0,\cdot)\right\|_{\mathrm L^2\(|v|^2\,d\gamma_k\)}\le\left\|f_0\right\|_{\mathrm L^2\(|v|^2\,d\gamma_k;\,\mathrm L^1(dx)\)}$. For some positive constants~$C$ and~$\lambda$, we get that
\be{eq:decayF}
\left\|\fav(t,\cdot)\right\|_{\mathrm L^2\(|v|^2\,d\gamma_k\)}^2\le C\,\left\|f_0\right\|_{\mathrm L^2\(|v|^2\,d\gamma_k;\,\mathrm L^1(dx)\)}^2\,e^{-\lambda\,t}\,, \quad\forall\,t\ge0\,.
\ee

%%%%%%%%%%%%%%%%%%%%%%%%%%%%%%%%%%%%%%%%%%%%%%%%%%%%%%%%%%%%%%%%%%%%%%
\subsubsection*{$\bullet$ Step 2: Improved decay of $f$}
Let us define $g(t,x,v):=f(t,x,v)-\fav(t,v)\,\varphi(x)$, where $\varphi$ is a given positive function satisfying
\[
\int_{\R^d} \varphi(x)dx = 1 \,,\qquad\mbox{e.g.}\quad \varphi(x):=(2\pi)^{-d/2}\,e^{-|x|^2/2}\,,\quad\forall\,x\in\R^d\,.
\]
Since $\partial_t \fav=\mathsf{L}\fav$, the Fourier transform $\hat g(t,\xi,v)$ of $g(t,x,v)$ solves
\[
\partial_t\hat g+\mathsf T\hat g=\mathsf L\hat g-\fav\,\mathsf T\hat\varphi \,,
\]
where $\mathsf T\hat\varphi=i\,(v\cdot\xi)\,\hat\varphi$. Using Duhamel's formula
\[
\hat g=e^{(\mathsf L-\mathsf T)\,t}\hat g_0-\int_0^te^{(\mathsf L-\mathsf T)\,(t-s)}\,\fav(s,v)\,\mathsf T\hat\varphi(\xi)\,ds\,,
\]
Corollary~\ref{lec14-modelemm}, and Proposition~\ref{th:factorization}, for some generic constant $C>0$ which will change from line to line, we get
\begin{multline}\label{Duhamel}\tag{21} % not evry elegant but otherwise conflicting with hyperref
\left\|\hat g(t,\xi,\cdot)\right\|_{\mathrm L^2\(d\gamma_k\)}\le C\,e^{-\frac12\,\mu_\xi\,t}\,\left\|\hat g_0(\xi,\cdot)\right\|_{\mathrm L^2\(d\gamma_k\)}\\+C\int_0^te^{-\frac{\mu_\xi}2\,(t-s)}\left\|\fav(s,\cdot)\right\|_{\mathrm L^2\(|v|^2\,d\gamma_k\)}|\xi|\,|\hat\varphi(\xi)|\,ds\,.
\end{multline}
The key observation is $\hat g_0 (0,v)=0$, so that $\hat g_0(\xi,v)=\int_0^{|\xi|}\frac\xi{|\xi|}\cdot\nabla_\xi\hat g_0\big(\eta\,\frac\xi{|\xi|},v\big)\,d\eta$ yields
\[
|\hat g_0(\xi,v)|\le|\xi|\,\left\|\nabla_\xi\hat g_0(\cdot,v)\right\|_{\mathrm L^\infty(d\xi)}\le|\xi|\,\left\|g_0(\cdot,v)\right\|_{\mathrm L^1(|x|\,dx)}\quad\forall\,(\xi,v)\in\R^d\times \R^d\,.
\]
We know from~\eqref{eq:mu_xi} that $\mu_\xi=\Lambda\,|\xi|^2/(1+|\xi|^2)$. The first term of the r.h.s.~of~\eqref{Duhamel} can therefore be estimated for any $t\ge1$ by
\begin{multline*}
\(\int_{|\xi|\le1}\int_{\R^d}\left|e^{(\mathsf L-\mathsf T)\,t}\hat g_0\right|^2d\gamma_k\,d\xi\)^{1/2}\le\(\int_{\R^d}|\xi|^2\,e^{-\frac\Lambda2\,|\xi|^2\,t}\,d\xi\)^{1/2}\left\|g_0\right\|_{\mathrm L^2\(d\gamma_k;\,\mathrm L^1(|x|\,dx)\)}\\
\le\frac C{(1+t)^{1+\frac d2}}\,\left\|g_0\right\|_{\mathrm L^2\(d\gamma_k;\,\mathrm L^1(|x|\,dx)\)}\,,
\end{multline*}
which is the leading order term as $t\to\infty$, and we have that
\[
\int_{|\xi|>1}e^{-\mu_\xi\,t}\,\left\|\hat g_0(\xi,\cdot)\right\|_{\mathrm L^2\(d\gamma_k\)}^2\,d\xi\le C\,e^{-\frac{\Lambda}2t}\,\left\|g_0\right\|_{\mathrm L^2\(dx\,d\gamma_k\)}^2
\]
for any $t\ge0$, using the fact that $\mu_\xi\ge\Lambda/2$ when $|\xi|\ge1$ and Plancherel's formula.

Using~\eqref{eq:decayF}, the second term of the r.h.s.~of~\eqref{Duhamel} is estimated by
\begin{multline*}
\int_{\R^d}\(\int_0^te^{-\frac{\mu_\xi}2\,(t-s)}\left\|\fav(s,\cdot)\right\|_{\mathrm L^2\(|v|^2\,d\gamma_k\)}|\xi|\,|\hat\varphi(\xi)|\,ds\)^2d\xi\\
\le
C\,\left\|f_0\right\|_{\mathrm L^2\(|v|^2\,d\gamma_k;\,\mathrm L^1(dx)\)}^2\int_{\R^d}|\xi|^2\,|\hat\varphi(\xi)|^2\(\int_0^te^{-\frac{\mu_\xi}2\,(t-s)}\,e^{-\frac\lambda2s}\,ds\)^2\,d\xi\,.
\end{multline*}
On the one hand, we use the Cauchy-Schwarz inequality to get
\begin{multline*}
\int_{|\xi|\le1}|\xi|^2\,|\hat\varphi(\xi)|^2\(\int_0^te^{-\frac{\mu_\xi}2\,(t-s)}e^{-\frac\lambda2s}\,ds\)^2\,d\xi\\
\le\left\|\varphi\right\|_{\mathrm L^1(dx)}^2\int_{|\xi|\le1}|\xi|^2\(\int_0^te^{-\,\mu_\xi\,(t-s)}\,e^{-\frac\lambda2 s}\,ds\)\(\int_0^{t}e^{-\frac\lambda2 s}\,ds\)d\xi\hspace*{1.5cm}\\
\le\frac2\lambda\,\left\|\varphi\right\|_{\mathrm L^1(dx)}^2\int_0^t\(\int_{|\xi|\le1}|\xi|^2e^{-\frac{\Lambda}2\,|\xi|^2\,(t-s)}\,d\xi\)e^{-\frac\lambda2\,s}\,ds\le C_1\,t^{-\frac d2-1}+C_2\,e^{-\frac\lambda4\,t}\,,
\end{multline*}
where the last inequality is obtained by splitting the integral in $s$ on $(0,t/2)$ and $(t/2,t)$. On the other hand, using $\mu_\xi\ge\Lambda/2$ when $|\xi|\ge1$, we obtain
\[
\int_{|\xi|\ge1}|\xi|^2\,|\hat\varphi(\xi)|^2\(\int_0^te^{-\frac{\mu_\xi}2\,(t-s)}\,e^{-\frac\lambda2s}\,ds\)^2\,d\xi\le t^2\,e^{-\min\lbrace\Lambda/2,\lambda\rbrace\,t}\,\left\|\nabla\varphi\right\|_{\mathrm L^2(dx)}^2\, .
\]

By collecting all terms, we deduce that $\left\|g(t,\cdot,\cdot)\right\|_{\mathrm L^2\(dx\,d\gamma_k\)}^2$ is bounded by
\[
C\(\left\|g_0\right\|_{\mathrm L^2\(d\gamma_k;\,\mathrm L^1(|x|\,dx)\)}^2+\left\|f_0\right\|_{\mathrm L^2\((|v|^2\,d\gamma_k\,;\mathrm L^1(dx)\)}^2\)(1+t)^{-\left(1+\frac d2\right)}\,,
\]
for some constant $C>0$. Recalling that $f=g+\fav\,\varphi$, the proof of Theorem~\ref{th:whole-space-zero-average} is completed using~\eqref{eq:decayF}.

%%%%%%%%%%%%%%%%%%%%%%%%%%%%%%%%%%%%%%%%%%%%%%%%%%%%%%%%%%%%%%%%%%%%%%
\subsection{Improved decay rates with higher order cancellations}\label{Sec:ImprovedDecayRates}

We prove Theorem~\ref{th:whole-space-zero-average2}, which means that from now on we assume in Case (a) that $M$ is a normalized Gaussian \eqref{Gaussian}, and in Case (b) that $\sigma\equiv 1$. Moreover, the initial data satisfies~\eqref{ZeroAverage}, that is,
\[
\iint_{\R^d\times\R^d} f_0\, P \, dx\,dv = 0\quad\forall\,P \in \R_\ell[X,V]\,.
\]
For any $P \in \mathbb \R_\ell[X]$, let
\[
P[f](t,v):=\int_{\R^d}P(x)\,f(t,x,v)\,dx\,,
\]
so that $\int_{\R^d}P[f](0,v)\,dv=0$. 

In this section we use the notation $\lesssim_k$ to express inequalities up to a constant which depends on~$k$.

%%%%%%%%%%%%%%%%%%%%%%%%%%%%%%%%%%%%%%%%%%%%%%%%%%%%%%%%%%%%%%%%%%%%%%
\subsubsection*{$\bullet$ Step 1: Conservation of zero moments} For a solution $f$ of~\eqref{eq:model} we compute

\begin{multline*}
\frac{d}{dt} \iint_{\R^d\times\R^d} f(t,x,v)\, P(x,v)\, dx\,dv \\= - \iint_{\R^d\times\R^d} (v \cdot \nabla_x f)\, P \, dx\,dv + \iint_{\R^d\times\R^d} (\mathsf{L}f)\, P \, dx\,dv\\
= \iint_{\R^d\times\R^d} \left( v \cdot \nabla_x P \right) f \, dx\,dv + \iint_{\R^d\times\R^d} (\mathsf{L}f)\, P \, dx\,dv\,.
\end{multline*}
In Case (a) of a Fokker-Planck operator, we may write 
\begin{multline*}
\iint_{\R^d\times\R^d} (\mathsf{L}f)\, P \, dx\,dv = \iint_{\R^d\times\R^d} \frac{1}{M}\,\nabla_v \cdot \left( M\,\nabla_v P \right) f \, dx\,dv \\= \iint_{\R^d\times\R^d} \left( \Delta_v P - v \cdot \nabla_v P\right) f \, dx\,dv\,.
\end{multline*}
By definition of $\R_\ell[X,V]$, it turns out that $\Delta_v P - v \cdot \nabla_v P \in \R_\ell[X,V]$. For the scattering operator of Case (b), one has 
\begin{align*}
&\iint_{\R^d\times\R^d} (\mathsf{L}f)\, P \, dx\,dv\\ 
&= \iint_{\R^d\times\R^d} \left( \int_{\R^d} \left( M(v)\,f(t,x,v') - M(v')\,f(t,x,v)\right)\,dv' \right)\, P(x,v)\, dx\,dv \\
&= \iiint_{\R^d \times \R^d \times \R^d} \left( M(v)\,f(t,x,v') - M(v')\,f(t,x,v)\right) \, P(x,v)\,dx\,dv\,dv'\\
&= \iint_{\R^d\times\R^d} \left( \int_{\R^d} M(v)\, P(x,v)\,dv \right) f(t,x,v')\,dx\,dv' - \iint_{\R^d\times\R^d} f(x,v)P(x,v)\,dx\,dv\,.
\end{align*}
One can check that $\int_{\R^d} M(v)\, P(x,v)\,dv \in \R_\ell[X]$. Since also $v\cdot\nabla_x P\in \R_\ell[X,V]$, the evolution of moments of order lower or equal than $\ell$ is equivalent to a linear ODE of the form $\dot Y(t) = Q\,Y(t)$, where $Q$ is a matrix resulting from the previous computations. Consequently, if $Y(0)=0$ initially, it remains null for all times. 

%%%%%%%%%%%%%%%%%%%%%%%%%%%%%%%%%%%%%%%%%%%%%%%%%%%%%%%%%%%%%%%%%%%%%%
\subsubsection*{$\bullet$ Step 2: Decay of polynomial averages in space.} We claim that for any $j \le \ell$, there exists $\lambda > 0$ such that, for any $P \in \R_j[X]$ and $q \in \N$,
\be{eq:estP}
\left\|P[f](t,\cdot)\right\|_{\mathrm L^2\(d\gamma_{k+q}\)}
\lesssim_{j,q}\|f_0\|_{\mathrm L^2\(d\gamma_{k+q+2j};\,\mathrm L^1\((1+|x|^j) \, dx\)\)}
\,(1+t)^j \, e^{-\lambda\,t}\quad\forall\,t\ge0\,.
\ee
Let us prove it by induction.

\smallskip\noindent\emph{1. The case $j=0$.} Notice that $j=0$ means that $P$ is a real number and $P[f]=\fav$ as defined in 
\eqref{eq:fav}, up to a multiplication by a constant. Since $\int_{\R^d} \fav(t,v) \, dv=0$ for any $t\ge0$, one has 
$\partial_t \fav = \mathsf{L}\fav$, thus we deduce from the \emph{microscopic coercivity property} as above that
\[
\|\fav(t,\cdot)\|_{\mathrm L^2(d\gamma_\infty)} \le\|\fav(0,\cdot)\|_{\mathrm L^2(d\gamma_\infty)}\,e^{-\lambda_m\,t}\quad\forall\,t\ge0\,.
\]
We also obtain that
\be{Ineq:j=0}
\|\fav(t,\cdot)\|_{\mathrm L^2\(d\gamma_{k+q}\)}\lesssim_q\|f_0\|_{\mathrm L^2\( d\gamma_{k+q};\,\mathrm L^1(dx)\)}\,e^{-\lambda\,t} \quad\forall\,t\ge0\,,
\ee
but this requires some comments. The case $k\in(d,\infty)$ is covered by Corollary~\ref{th:fac-appl}. 

The case $k=\infty$ in~\eqref{Ineq:j=0} is given by the following lemma.
%---------------------------------------------------------------------
\begin{lemma} Under the assumptions of Theorem~\ref{th:whole-space-zero-average2}, one has
\begin{equation*}
\|\fav(t,\cdot)\|_{\mathrm L^2\((1+|v|^q)\,d\gamma_\infty\)}\lesssim_q\|f_0\|_{\mathrm L^2\((1+|v|^q)\,d\gamma_\infty;\,\mathrm L^1(dx)\)}\,e^{-\lambda\,t} \quad\forall\,t\ge0\,.
\end{equation*}
\end{lemma}
%---------------------------------------------------------------------
\begin{proof}
We rely on Proposition~\ref{prop:shrink} with the Banach spaces $\mathcal B_1 = \sfL^2(d\gamma_\infty)$ and $\mathcal B_2 = \sfL^2\big((1+|v|^q)\,d\gamma_\infty\big)$. In Case~(a), let us define $\mathfrak A$ and $\mathfrak B$ by $\mathfrak AF=N\,\chi_RF$ and $\mathfrak BF=\mathsf LF-\mathfrak AF$.
In Case~(b), we consider $\mathfrak A$ and $\mathfrak B$ such that
\begin{eqnarray*}
\mathfrak AF(v)&=&M(v)\int_{\R^d} F(v')\,dv'\,,\\
\mathfrak BF(v)&=&-\int_{\R^d} M(v')\,dv'\,F(v)\,.
\end{eqnarray*}
The semi-group generated by $\mathfrak A +\mathfrak B$ is exponentially decreasing in $\mathcal B_1$ by the microscopic coercivity property, as above.
The semi-group generated by $\mathfrak B$ is exponentially decreasing in $\mathcal B_2$. In Case (b), it is straightforward. In Case (a), $F(t)=e^{\mathfrak B t}\,F_0$ is such that
\begin{align*}
\frac12\,&\frac d{dt}\int_{\R^d}|F|^2\(1+|v|^q\)d\gamma_\infty=\int_{\R^d}(\mathfrak BF)\,F\(1+|v|^q\)d\gamma_\infty\\
& =\int_{\R^d}\nabla_v \left(M\,\nabla_v\left( \tfrac{F}{M}\right)\right)\,F\(1+|v|^q\)\,d\gamma_\infty-\int_{\R^d} N \chi_R(v)\,| F |^2 \(1+|v|^q\)d\gamma_\infty \\
& = - \int_{\R^d} \left| \nabla_v\left( \tfrac{F}{M}\right) \right|^2\, \(1+|v|^q\)M\,dv - \int_{\R^d} q\, | v |^{q-2}\,v \cdot \nabla_v\left( \tfrac{F}{M}\right) \tfrac{F}{M}\,M\,dv \\ & \hspace*{6cm} - \int_{\R^d} N \chi_R(v)\,| F |^2 \(1+|v|^q\)\,\frac{dv}{M} \\
&\le \int_{\R^d} \left\lbrace \tfrac q{2}\,\tfrac{ \nabla_v \cdot \left( | v |^{q-2}\,v\, M \right) }{\(1+|v|^q\)M} - N \chi_R(v) \right\rbrace | F | ^2 \(1+|v|^q\)\,\frac{dv}{M}\le-\frac \lambda2\int_{\R^d}|F|^2\(1+|v|^q\)d\gamma_\infty
\end{align*}
for some $\lambda>0$, by choosing $N$ and $R$ large enough.

The operator $\mathfrak{A}:\mathcal{B}_1\to \mathcal{B}_2$ is bounded. This is straightforward in Case (a) and follows from the boundedness of $\int_{\R^d}M(v)\,\(1+|v|^q\)d\gamma_\infty$ in Case~(b). Proposition~\ref{prop:shrink} applies which concludes the proof.
\end{proof}

\smallskip\noindent\emph{2. Induction.} Let us assume that~\eqref{eq:estP} is true for some $j \ge 0$, consider $P \in
\mathbb \R_{j+1}[X]$ and observe that $P[f]$ solves 
\begin{equation*}
\partial_t P[f] = \mathsf L P[f] - \int_{\R^d} \left( v \cdot\nabla_x P\right) f \,dx\,. 
\end{equation*}
Since $\nabla_x P \in \R_j[X]$, the
induction hypothesis at step $j$ (applied with $q$ replaced by $q+2$) gives
\begin{align*}
\textstyle\hspace*{-12pt}\left\| v \cdot \int_{\R^d}\left( \nabla_x P \right)[f]\,dx \right\|_{\mathrm L^2\(d\gamma_{k+q}\)} &\textstyle\lesssim \left\| \int_{\R^d}\left( \nabla_x P \right)[f]\,dx \right\|_{\mathrm L^2\(\,d\gamma_{k+q+2}\)} \\ & \lesssim_{j,q}\|f_0\|_{\mathrm L^2\(d\gamma_{k+q+2(j+1)};\,\mathrm L^1\((1+|x|^j) \,
dx\)\)}\,(1+t)^j \, e^{-\lambda\,t}\,.
\end{align*}
By Duhamel's formula, we have 
\begin{equation*}
P[f](t,v) = e^{\mathsf Lt} P[f](0,v) - \int_0^t e^{\mathsf L(t-s)} \left( v \cdot \int_{\R^d} \left(\nabla_x P\right)[f_s] \,dx \right) \, ds\,.
\end{equation*}
Note that $\int_{\R^d} v \cdot \int_{\R^d} \left(\nabla_x P\right)[f] \,dx\,dv = \iint_{\R^d\times\R^d} \left(v \cdot \nabla_x P\right)[f] \,dx\,dv= 0$ for all $t\ge0$ since $v \cdot \nabla_x P \in \R_\ell[X,V]$. As a consequence, the decay of the semi-group associated with $\mathsf L$ can be estimated by 
\begin{multline*}
\left\| e^{\mathsf L(t-s)} \left( v \cdot \int_{\R^d} \left(\nabla_x P\right)[f_s] \,dx \right) \right\|_{\mathrm L^2(d\gamma_\infty)} \le \left\|v \cdot \int_{\R^d} \left(\nabla_x P\right)[f_s] \,dx \right\|_{\mathrm L^2(d\gamma_\infty)}\,e^{-\lambda_m\,(t-s)}\,.
\end{multline*}
As in the case $j=0$, we deduce from Corollary~\ref{th:fac-appl} that
\begin{multline*}
\left\| e^{\mathsf L(t-s)} \left( v \cdot \int_{\R^d} \left(\nabla_x P\right)[f_s] \,dx \right) \right\|_{\mathrm L^2((1+|v|^q)\,d\gamma_k)}\\
\le \left\|v \cdot \int_{\R^d} \left(\nabla_x P\right)[f_s] \,dx \right\|_{\mathrm L^2(d\gamma_{k+q})}\,e^{-\lambda\,(t-s)}\\
\lesssim_{q,k}\|f_0\|_{\mathrm L^2\(d\gamma_{k+q+2(j+1)};\,\mathrm L^1\((1+|x|^j) \,
dx\)\)}\,(1+s)^j \, e^{-\lambda\,t}\,.
\end{multline*}
Moreover, since $\iint_{\R^d\times\R^d} f_0(x,v)\, P(x) \, dx\,dv =0$, for the same reasons we also have that
\begin{equation*}
\Big\| e^{\mathsf Lt} P[f](0,\cdot) \Big\|_{\mathrm L^2\(d\gamma_{k+q}\)} \le \left\|P[f_0] \right\|_{\mathrm L^2\((1+|v|^q)\,d\gamma_k\)}\,e^{-\lambda\,t}
\end{equation*}
for some $\lambda>0$. We deduce from Duhamel's formula that
\begin{align*}
&\| P[f]\|_{\mathrm L^2\(d\gamma_{k+q}\)}\\
&\lesssim\left\| e^{\mathsf Lt} P[f](0,\cdot)\right\|_{\mathrm L^2\(d\gamma_{k+q}\)}+ \int_0 ^t{\textstyle \left\| e^{-\mathsf L \, (t-s)} \left( v
\cdot \int_{\R^d}\nabla_x P[f_s]\,dx \right) \right\|_{\mathrm L^2\(d\gamma_{k+q}\)}}\, ds \\ &\lesssim_k\|f_0\|_{\mathrm L^2\(d\gamma_{k+q};\,\mathrm L^1\((1+|x|^{j+1}) \,
dx\)\)} \,e^{-\lambda\,t}\\
&\hspace*{4cm}+ \int_0 ^t (1+s)^j \, e^{-\lambda
\,t} \,\|
f_0\|_{\mathrm L^2\(d\gamma_{k+q+2(j+1)};\,\mathrm L^1\((1+|x|^{j}) \, dx\)\)}\, ds \\ 
&\lesssim_k\|f_0\|_{\mathrm L^2\(d\gamma_{k+q+2(j+1)};\,\mathrm L^1\((1+|x|^{j+1}) \,
dx\)\)} \,(1+t)^{j+1} e^{-\lambda\,t}\,,
\end{align*}
which proves the induction.

%%%%%%%%%%%%%%%%%%%%%%%%%%%%%%%%%%%%%%%%%%%%%%%%%%%%%%%%%%%%%%%%%%%%%%
\subsubsection*{$\bullet$ Step 3: Improved decay of $f$.} Let us choose some $t_0>0$. In order to estimate $\big\|f(t,\cdot,\cdot) \big\|_{\mathrm L^2(dx \, d\gamma_k)} ^2=\big\| e^{(\mathsf L-\mathsf T)t} f_0 \big\|_{\mathrm L^2(dx \, d\gamma_k)} ^2$, we compute its evolution on $(0,2\,t_0)$ and split the interval on $(0,t_0)$ and $(t_0,2\,t_0)$ using the semi-group property
\[
\left\| e^{(\mathsf L-\mathsf T)\,(2\,t_0)} f_0 \right\|_{\mathrm L^2(dx \, d\gamma_k)} ^2 = \left\| e^{(\mathsf L-\mathsf T)\,t_0} \left( e^{(\mathsf L-\mathsf T)\,t_0} f_0 \right) \right\|_{\mathrm L^2(dx \, d\gamma_k)} ^2\,.
\]
Up to the end of this section, $\mathsf T=v\cdot\nabla_x$ denotes the transport operator in position and velocity variables. We decompose $f_{t_0} = e^{(\mathsf L-\mathsf T)\,t_0} f_0$ into
\[
f_{t_0} = \(\sum_{| \alpha | \le \ell} \frac{1}{\alpha !}\,X^\alpha[f_{t_0}] \, \partial^{\alpha} \varphi\) + g_0
\quad\mbox{with}\quad g_0 := f_{t_0} - \sum_{| \alpha | \le \ell} \frac{1}{\alpha !}\,X^\alpha[f_{t_0}] \, \partial^{\alpha} \varphi
\]
where $\alpha=(\alpha_1,\alpha_2,\ldots\alpha_i\ldots\alpha_d)\in\N^d$ is a multi-index such that $|\alpha|=\sum_{i=1}^d\alpha_i\le\ell$ and $\varphi$ is given by
\[
\varphi(x):=(2\pi)^{-d/2}\,e^{-|x|^2/2}\quad\forall\,x\in\R^d\,.
\]
Here we use the notation $\partial^{\alpha}\varphi=\partial_{x_1}^{\alpha_1}\partial_{x_2}^{\alpha_2}\ldots\partial_{x_d}^{\alpha_d}\varphi$ and $X^\alpha=\prod_i^n X_i^{\alpha_i}$.
According to~\eqref{eq:estP}, we know that
\begin{equation*}
\left\| X^\alpha[f_{t_0}] \right\|_{\mathrm L^2( d\gamma_k)} \lesssim_{j}\|f_0\|_{\mathrm L^2\(d\gamma_{k+2j};\,\mathrm L^1\((1+|x|^j) \, dx\)\)}
\,\left(1+t_0\right)^j \, e^{-\lambda\,t_0}\,,
\end{equation*}
so that, by considering the evolution of the first term on $(t_0,2\,t_0)$, we obtain
\be{EstimImpr1}
\left\| e^{(\mathsf L-\mathsf T)\,t_0}\!\( \sum_{| \alpha | \le \ell} \frac{1}{\alpha !}\,X^\alpha[f_{t_0}] \, \partial^{\alpha} \varphi \) \right\|_{\mathrm L^2(dx \, d\gamma_k)}\kern-10pt\lesssim
\sum_{| \alpha | \le \ell} \left\| X^\alpha[f_{t_0}] \right\|_{\mathrm L^2( d\gamma_k)} \, \left\|\partial^{\alpha} \varphi \right\|_{\mathrm L^2(dx)}\lesssim e^{-\frac\lambda2\,t_0}\,.
\ee
Next, let us consider the second term and define, on $t+t_0\in(t_0,2\,t_0)$, the function
\[
g:= f_{t+t_0} - \sum_{| \alpha | \le \ell} \frac{1}{\alpha !}\,X^\alpha[f_{t+t_0}] \, \partial^{\alpha} \varphi\,.
\]
With initial datum $g_0$, it solves on$(0,t_0)$ the equation
\begin{align*}
\partial_t g &= \partial_t f_{t+t_0} - \sum_{| \alpha | \le \ell} \frac{1}{\alpha !}\,\partial_t\left( X^\alpha[f_{t+t_0}]\right)\, \partial^{\alpha} \varphi \\
& = (\sfL - \sfT)(f_{t+t_0}) - \sfL\left( \sum_{| \alpha | \le \ell} \frac{1}{\alpha !}\,X ^\alpha[f_{t+t_0}]\,\partial^{\alpha} \varphi \right)\\
&\hspace*{4cm} + \sum_{| \alpha | \le \ell} \frac{1}{\alpha !}\,\left( \int_{\R^d} (v \cdot \nabla_x x^\alpha)\,f_{t+t_0} \, dx\right)\, \partial^{\alpha} \varphi \\
& = (\sfL - \sfT)(g) - \sfT\left( \sum_{| \alpha | \le \ell} \frac{1}{\alpha !}\,X ^\alpha[f_{t+t_0}]\,\partial^{\alpha} \varphi \right) + \sum_{| \alpha | \le \ell} \frac{1}{\alpha !}\,\left( \int_{\R^d} (v \cdot \nabla_x x^\alpha)\,f_{t+t_0} \, dx\right)\, \partial^{\alpha} \varphi \\
& = (\sfL - \sfT)(g) + v \cdot\sum_{| \alpha | \le \ell} \frac{1}{\alpha !}\,\left(\nabla_xX^\alpha[f]\,\partial^{\alpha}\varphi - X ^\alpha[f_{t+t_0}]\,\nabla_x(\partial^{\alpha} \varphi) \right)
\end{align*}
where $\alpha!=\prod_{i=1}^d\alpha_i!$ is associated with the multi-index $\alpha=(\alpha_i)_{i=1}^d$ and
\[
\nabla_xX^\alpha[f]=\(\partial_{x_i}X^\alpha[f]\)_{i=1}^d:=\(\int_{\R^d}\partial_{x_i}x^\alpha\,f\,dx\)_{i=1}^d=\(\int_{\R^d}\alpha_i\,x^{\alpha_{\wedge i}}\,f\,dx\)_{i=1}^d\,,
\]
Here the notation $\alpha_{\wedge i}$ denotes the multi-index $(\alpha_1,\alpha_2\ldots\alpha_{i-1},\alpha_i-1,\alpha_{i+1}\ldots\alpha_d)$ with the convention that $X^{\alpha_{\wedge i}}\equiv0$ if $\alpha_i=0$. We also define the opposite transformation $\alpha_{\vee i}:=(\alpha_1,\alpha_2\ldots\alpha_{i-1},\alpha_i+1,\alpha_{i+1}\ldots\alpha_d)$ so that $\partial_{x_i}(\partial^{\alpha} \varphi)=\partial^{\alpha_{\vee i}} \varphi$. Let us consider the last term and start with the case $d=1$. In that case, 
\begin{multline*}
v \cdot\sum_{| \alpha | \le \ell} \frac{1}{\alpha !}\,\left( \nabla_xX^\alpha[f]\,\partial^{\alpha}\varphi - X ^\alpha[f_{t+t_0}]\,\nabla_x(\partial^{\alpha} \varphi) \right)\\
=v_1\sum_{\alpha_1=0}^\ell \frac{1}{\alpha_1!}\,\left( \(\int_{\R}\(\alpha_1\,x^{\alpha_1-1}\)\,f_{t+t_0} \, dx\)\,\partial x_1^{\alpha_1}\varphi-\(\int_{\R}x^{\alpha_1}\,f_{t+t_0} \, dx\)\,\partial x_1^{\alpha_1+1}\varphi\right)\\
=-\,\frac{v_1}{\ell!}\(\int_{\R}x^\ell\,f_{t+t_0} \, dx\)\,\partial x_1^{\ell+1}\varphi
\end{multline*}
because it is a telescoping sum. We adopt the convention that $\alpha!=1$ if $\alpha_i\le0$ for some $i=1,2\ldots d$. The same property holds in higher dimensions:
\begin{multline*}
\sum_{| \alpha | \le \ell} \frac{1}{\alpha !}\,\left(\partial_{x_i}X^\alpha[f]\,\partial^{\alpha}\varphi - X ^\alpha[f_{t+t_0}]\,\partial_{x_i}(\partial^{\alpha} \varphi) \right)\\
=\sum_{| \alpha | \le \ell}\(\frac{1}{\alpha_{\wedge i} !}\,X^{\alpha_{\wedge i}}[f]\,\partial^{\alpha}\varphi - \frac{1}{\alpha!}\,X ^\alpha[f_{t+t_0}]\,\partial^{\alpha_{\vee i}} \varphi \)
=- \sum_{| \alpha | = \ell} \frac{1}{\alpha !}\,X ^\alpha[f_{t+t_0}]\,\partial_{x_i}(\partial^{\alpha} \varphi)\,.
\end{multline*}
We deduce that
\begin{align*}
\partial_t g = (\sfL - \sfT)(g) - v \cdot\sum_{| \alpha | = \ell} \frac{1}{\alpha !}\,X ^\alpha[f_{t+t_0}]\,\nabla_x(\partial^{\alpha} \varphi)\,.
\end{align*}
Duhamel's formula in Fourier variables gives
\[
\hat g(t_0,\xi,v)=e^{(\mathsf L-\mathsf T)\,t_0}\hat g_0-\int_0^{t_0}e^{(\mathsf L-\mathsf T)\,\left(t_0-s\right)}\,\left( v \cdot\sum_{| \alpha | = \ell} \frac{1}{\alpha !}\,X ^\alpha[f_{s+t_0}]\,\widehat{\nabla_x(\partial^{\alpha} \varphi) }\right)\,ds
\]
up to a straightforward abuse of notations. Hence
\begin{multline*}\label{Duhamel2}
\left\|\hat g(t_0,\xi,\cdot)\right\|_{\mathrm L^2\(d\gamma_k\)}\lesssim e^{-\frac12\,\mu_\xi\,t_0}\,\left\|\hat g_0(\xi,\cdot)\right\|_{\mathrm L^2\(d\gamma_k\)}\\+\int_0^{t_0}e^{-\frac{\mu_\xi}2\,(t_0-s)} \sum_{| \alpha | = \ell} \frac{1}{\alpha !}\,\left\| X ^\alpha[f_{s+t_0}] \right\|_{\mathrm L^2\(|v|^2\,d\gamma_k\)}\,\big| \widehat{\nabla_x(\partial^{\alpha} \varphi) }\big|\,ds\,.
\end{multline*}
Recall that \eqref{eq:estP} gives 
\begin{equation*}
\left\|X ^\alpha[f_{s+t_0}]\right\|_{\mathrm L^2\(|v|^{2}\,d\gamma_k\)} \lesssim_{\ell}\|f_0\|_{\mathrm L^2\(d\gamma_{k+2\ell+2};\,\mathrm L^1\((1+|x|^\ell) \, dx\)\)}
\, e^{-\frac{\lambda}{2}\,s}\,.
\end{equation*}
On the other hand we use $\big| \widehat{\nabla_x(\partial^{\alpha} \varphi) }\big| \le | \xi |^{\ell+1} | \,\hat \varphi |$ and observe that 
\begin{equation*}
|\hat g_0(\xi,v)|\lesssim |\xi|^{\ell+1} \left\|g_0(\cdot,v)\right\|_{\mathrm L^1(|x|^\ell\,dx)}\quad\forall\,(\xi,v)\in\R^d\times \R^d\,.
\end{equation*}
Collecting terms, we have that
\begin{multline*}
\left\|\hat g(t_0,\xi,\cdot)\right\|_{\mathrm L^2\(d\gamma_k\)}\\
\lesssim e^{-\frac12\,\mu_\xi\,t_0}\,|\xi|^{\ell+1}\,\mathbf 1_{|\xi|<1}\, \left\|g_0(\cdot,v)\right\|_{\mathrm L^2\(d\gamma_k;\,\mathrm L^1(|x|^\ell\,dx)\)}+e^{-\frac12\,\mu_\xi\,t_0}\,\mathbf 1_{|\xi|\ge1}\,\left\|\hat g_0(\xi,\cdot)\right\|_{\mathrm L^2\(d\gamma_k\)}\\
+|\xi |^{\ell+1} | \,\hat \varphi(\xi)|\,\|f_0\|_{\mathrm L^2\(d\gamma_{k+2\ell+2};\,\mathrm L^1\((1+|x|^\ell) \, dx\)\)}\int_0^{t_0}e^{-\frac{\mu_\xi}2\,(t_0-s)}\,e^{-\frac{\lambda}{2}\,s}\,ds\,.
\end{multline*}

We know from~\eqref{eq:mu_xi} that $\mu_\xi=\Lambda\,|\xi|^2/(1+|\xi|^2)$ so that $\mu_\xi\ge\frac\Lambda2\,|\xi|^2$ if $|\xi|<1$ and $\mu_\xi\ge\Lambda/2$ if $|\xi|\ge1$. Hence, for any $t_0\ge1$,
\[
\left\|e^{-\frac12\,\mu_\xi\,t_0}\,|\xi|^{\ell+1}\,\mathbf 1_{|\xi|<1}\right\|_{\mathrm L^2(d\xi)}\le\(\int_{\R^d}e^{-\frac\Lambda2\,|\xi|^2\,t_0}\,|\xi|^{2(\ell+1)}\,d\xi\)^{1/2}\lesssim t_0^{-\,(1+\ell+\frac d2)}\,,
\]
\[
\int_{|\xi|\ge1}e^{-\mu_\xi\,t_0}\,\left\|\hat g_0(\xi,\cdot)\right\|_{\mathrm L^2\(d\gamma_k\)}^2\,d\xi\lesssim e^{-\frac\Lambda2\,t_0}\,\left\|g_0\right\|_{\mathrm L^2\(dx\,d\gamma_k\)}^2
\]
by Plancherel's formula. We conclude by observing that
\begin{align*}
\int_{|\xi|\le1}|\xi|^{ \ell+1 }&\,|\hat\varphi(\xi)| \int_0^{t_0}e^{-\frac{\mu_\xi}2\,(t_0-s)}e^{-\frac\lambda2s}\,ds\;d\xi\\
&\le\left\|\varphi\right\|_{\mathrm L^1(dx)} \int_0^{t_0} \(\int_{|\xi|\le1}|\xi|^{\ell + 1} e^{-\,\frac{\Lambda}2\,|\xi|^2\,(t_0-s)} \, d\xi\)e^{-\frac\lambda2 s}\,ds\lesssim t_0^{-\,(1+\ell+\frac d2)}\,,\\
\int_{|\xi|\ge1}|\xi|^{ \ell+1 }&\,|\hat\varphi(\xi)| \int_0^{t_0}e^{-\frac{\mu_\xi}2\,(t_0-s)}\,e^{-\frac\lambda2s}\,ds \,d\xi\lesssim \big\|\,|\xi|^{\ell+1}\,\hat\varphi(\xi)\big\|_{\mathrm L^1(d\xi)}\,
t_0\,e^{-\frac14\,\min\lbrace\Lambda,2\lambda\rbrace\,t_0}\,.
\end{align*}
Altogether, we obtain that
\[
\left\|g(t_0,\cdot,\cdot)\right\|_{\mathrm L^2\(dx\,d\gamma_k\)}^2=\left\|\hat g(t_0,\cdot,\cdot)\right\|_{\mathrm L^2\(d\xi\,d\gamma_k\)}^2\lesssim t_0^{-\,(1+\ell+\frac d2)}\,.
\]
The decay result of Theorem~\ref{th:whole-space-zero-average2} is then obtained by writing
\[
\left\|f_{2t_0}\right\|_{\mathrm L^2\(dx\,d\gamma_k\)}^2\lesssim \left\|g(t_0,\cdot,\cdot)\right\|_{\mathrm L^2\(dx\,d\gamma_k\)}^2+\left\| e^{(\mathsf L-\mathsf T)\,t_0} \( \sum_{| \alpha | \le \ell} \frac{1}{\alpha !}\,X^\alpha[f_{t_0}] \, \partial^{\alpha} \varphi \) \right\|_{\mathrm L^2(dx \, d\gamma_k)}
\]
and using~\eqref{EstimImpr1} for any $t_0\ge1$, with $t=2\,t_0$. For $t\le2$, the estimate of Theorem~\ref{th:whole-space-zero-average2} is straightforward by Corollary~\ref{th:fac-appl}, which concludes the proof.

%%%%%%%%%%%%%%%%%%%%%%%%%%%%%%%%%%%%%%%%%%%%%%%%%%%%%%%%%%%%%%%%%%%%%%
%%%%%%%%%%%%%%%%%%%%%%%%%%%%%%%%%%%%%%%%%%%%%%%%%%%%%%%%%%%%%%%%%%%%%%
\appendix\renewcommand{\thesection}{Appendix~\Alph{section}}
\section{An explicit computation of Green's function for the kinetic
Fokker-Planck equation and consequences}\label{Sec:Green}

In the whole space case, when $M$ is the normalized Gaussian function, let us consider the kinetic Fokker-Planck equation of Case~(a)
\be{Eqn:kFP}
\partial_tf+v\cdot\nabla_xf=\nabla_v\cdot(v\,f+\nabla_vf)
\ee
on $(0,\infty)\times\R^d\times\R^d\ni(t,x,v)$. The characteristics associated with the equations
\[
\frac{dx}{dt}=v\,,\quad\frac{dv}{dt}=-\,v
\]
suggest to change variables and consider the distribution function $g$ such that
\[
f(t,x,v)=e^{d\,t}\,g\(t,x+\big(1-e^t\big)\,v,e^t\,v\)\quad\forall\,(t,x,v)\in(0,\infty)\times\R^d\times\R^d\,.
\]
The kinetic Fokker-Planck equation is changed into a heat equation in both variables $x$ and $v$ with $t$ dependent coefficients, which can be written as
\be{Eqn:tFP}
\partial_tg=\nabla\cdot\dot{\mathcal D}\,\nabla g
\ee
where $\nabla g=(\nabla_vg,\nabla_xg)$ and $\dot{\mathcal D}$ is the $t$-derivative of the bloc-matrix
\[
\mathcal D=\tfrac12\(\begin{array}{cc}
\mathsf a\,\mathrm{Id}&\mathsf b\,\mathrm{Id}\\
\mathsf b\,\mathrm{Id}&\mathsf c\,\mathrm{Id}
\end{array}\)
\]
with $\mathsf a=e^{2t}-1$, $\mathsf b=2\,e^t-1-e^{2t}$, and $\mathsf c=e^{2t}-4\,e^t+2\,t+3$. Here $\mathrm{Id}$ is the identity matrix on $\R^d$. We observe that $\dot{\mathcal D}$ is degenerate: it is nonnegative but its lowest eigenvalue is $0$. However, the change of variables allows the computation of a Green function.
%--------------------------------------------------------------------
\begin{lemma}\phantomsection\label{lem:kFPGreen} The Green function of~\eqref{Eqn:tFP} is given for any $(t,x,v)\in(0,\infty)\times\R^d\times\R^d$ by
\[
G(t,x,v)=\frac1{\big(2\pi\,(\mathsf a\,\mathsf c-\mathsf b^2)\big)^{d/2}}\,\exp\(-\frac{\mathsf a\,|x|^2-2\,\mathsf b\,x\cdot v+\mathsf c\,|v|^2}{2\,(\mathsf a\,\mathsf c-\mathsf b^2)}\)\,.
\]
\end{lemma}
%--------------------------------------------------------------------
\noindent The method is standard and goes back to~\cite{MR1503147} (also see~\cite{MR0168018,MR0222474} and~\cite{MR1052014,MR1200643}).
\begin{proof} By a Fourier transformation in $x$ and $v$, with associated variables $\xi$ and~$\eta$, we find that
\begin{multline*}
\log C-\log\hat G(t,\xi,\eta)=(\eta,\xi)\cdot\mathcal D(\eta,\xi)=\tfrac12\(\mathsf a\,|\eta|^2+2\,\mathsf b\,\eta\cdot\xi+\mathsf c\,|\xi|^2\)\\
=\tfrac12\,\mathsf a\,\left|\eta+\tfrac{\mathsf b}{\mathsf a}\,\xi\right|^2+\tfrac12\,\mathsf A\,|\xi|^2\,,\quad\mathsf A=\mathsf c-\tfrac{\mathsf b^2}{\mathsf a}
\end{multline*}
for some constant $C>0$ which is determined by the mass normalization condition $\nrm{G(t,\cdot\,,\cdot)}{\mathrm L^1(\R^d\times\R^d)}=1$. Let us take the inverse Fourier transform with respect to~$\eta$,
\begin{multline*}
(2\pi)^{-d}\int_{\R^d}e^{ iv\cdot\eta}\,\hat G(t,\xi,\eta)\,d\eta=\frac C{(2\pi\,\mathsf a)^{d/2}}\,e^{-\frac{|v|^2}{2\,\mathsf a}-i\tfrac{\mathsf b}{\mathsf a}\,v\cdot\xi}\,e^{-\tfrac12\,\mathsf A\,|\xi|^2}\\
=\frac C{(2\pi\,\mathsf a)^d}\,e^{-\frac{|v|^2}{2\,\mathsf a}}\,e^{-\tfrac12\,\mathsf A\,\left|\xi+i\tfrac{\mathsf b}{\mathsf a\,\mathsf A}\,v\right|^2-\tfrac{\mathsf b^2}{2\,\mathsf a^2\,\mathsf A}\,|v|^2}\,,
\end{multline*}
and then the inverse Fourier transform with respect to $\xi$, so that we obtain
\[
G(t,x,v)=\frac C{(2\pi\,\mathsf a)^\frac d2\,(2\pi\,\mathsf A)^\frac d2}\,e^{-\big(1+\frac{\mathsf b^2}{\mathsf a\,\mathsf A}\big)\,\frac{|v|^2}{2\,\mathsf a}}\,e^{-\frac{|x|^2}{2\,\mathsf A}}\,e^{\frac{\mathsf b}{\mathsf a\,\mathsf A}\,x\cdot v}=\frac C{(4\pi^2\,\mathsf a\,\mathsf A)^\frac d2}\,e^{-\frac1{2\,\mathsf A}\,{\left|x-\tfrac{\mathsf b}{\mathsf a}\,v\right|^2}}\,e^{-\frac{|v|^2}{2\,\mathsf a}}\,.
\]
It is easy to check that $C=1$.\end{proof}

Let us consider a solution $g$ of~\eqref{Eqn:tFP} with initial datum $g_0\in\mathrm L^1(\R^d\times\R^d)$. From the representation
\[
g(t,\cdot,\cdot)=G(t,\cdot,\cdot)*_{x,v}g_0\,,
\]
we obtain the estimate
\begin{multline*}
\nrm{g(t,\cdot,\cdot)}{\mathrm L^\infty(\R^d\times\R^d)}\le\nrm{G(t,\cdot,\cdot)}{\mathrm L^\infty(\R^d\times\R^d)}\,\nrm{g_0}{\mathrm L^1(\R^d\times\R^d)}\\
=\frac{\nrm{g_0}{\mathrm L^1(\R^d\times\R^d)}}{\(8\,\pi^2\)^{d/2}}\,t^{-\frac d2}\,e^{-\,d\,t}\,\(1+O\(t^{-1}\)\)
\end{multline*}
as $t\to\infty$. As a consequence, we obtain that the solution of~\eqref{Eqn:kFP} with a nonnegative initial datum $f_0$ satisfies
\[
\nrm{f(t,\cdot,\cdot)}{\mathrm L^\infty(\R^d\times\R^d)}=\frac{\nrm{f_0}{\mathrm L^1(\R^d\times\R^d)}}{\(8\,\pi^2\,t\)^{d/2}}\,\big(1+o(1)\big)\quad\mbox{as}\quad t\to\infty\,.
\]
Using the simple H\"older interpolation inequality
\[
\nrm f{\mathrm L^p(\R^d\times\R^d)}\le\nrm f{\mathrm L^1(\R^d\times\R^d)}^{1/p}\,\nrm f{\mathrm L^\infty(\R^d\times\R^d)}^{1-1/p}\,,
\]
we obtain the following decay result.
%--------------------------------------------------------------------
\begin{corollary}\phantomsection\label{cor:kFPdecay} If $f$ is a solution of~\eqref{Eqn:kFP} with a nonnegative initial datum $f_0\in\mathrm L^1(\R^d\times\R^d)$, then for any $p\in(1,\infty]$ we have the decay estimate
\[
\nrm{f(t,\cdot,\cdot)}{\mathrm L^p(\R^d\times\R^d)}\le\frac{\nrm{f_0}{\mathrm L^1(\R^d\times\R^d)}}{\(8\,\pi^2\,t\)^{\frac d2\,\big(1-\frac1p\big)}}\,\big(1+o(1)\big)\quad\mbox{as}\quad t\to\infty\,.
\]
\end{corollary}
%--------------------------------------------------------------------
\noindent By taking $f_0(x,v)=G(1,x,v)$, it is moreover straightforward to check that this estimate is optimal. With $p=2$, this also proves that the decay rate obtained in Theorem~\ref{th:whole-space-mode} for the Fokker-Planck operator, \emph{i.e.}, Case~(a), is the optimal one because, again with $f_0(x,v)=G(1,x,v)$, we observe that
\[
\left\|f(t,\cdot,\cdot)\right\|_{\mathrm L^2\(dx\,d\gamma_k\)}^2=e^{\,d\,t}\,\left\|G(t,\cdot,\cdot)\right\|_{\mathrm L^2\(dx\,dv\)}^2=O\(t^{-d/2}\)\quad\mbox{as}\quad t\to+\infty\,.
\]

%%%%%%%%%%%%%%%%%%%%%%%%%%%%%%%%%%%%%%%%%%%%%%%%%%%%%%%%%%%%%%%%%%%%%%
%%%%%%%%%%%%%%%%%%%%%%%%%%%%%%%%%%%%%%%%%%%%%%%%%%%%%%%%%%%%%%%%%%%%%%
\section{Consistency with the decay rates of the heat equation}
\renewcommand{\thesection}{\Alph{section}}
\label{sec:consistency}

In the whole space case, the abstract approach of~\cite{DMS-2part} is inspired by the diffusion limit of~\eqref{eq:model}. We consider the scaled equation
\be{rescaled}
\eps\,\frac{dF}{dt}+\mathsf TF=\frac1\eps\,\mathsf{L}F\,,
\ee
which formally corresponds to a parabolic rescaling given by $t\mapsto\eps^2\,t$ and $x\mapsto\eps\,x$, and investigate the limit as $\eps\to0_+$. Let us check that the rates are asymptotically independent of $\eps$ and consistent with those of the heat equation.

%%%%%%%%%%%%%%%%%%%%%%%%%%%%%%%%%%%%%%%%%%%%%%%%%%%%%%%%%%%%%%%%%%%%%%
\subsection{Mode-by-mode hypocoercivity}
It is straightforward to check that in the estimate~\eqref{decay-const} for $\lambda$, the gap constant $\lambda_m$ has to be replaced by $\lambda_m/\eps$ while, with the notations of Proposition~\ref{theo:DMS}, $C_M$ can be replaced by $C_M/\eps$ for $\eps<1$. In the asymptotic regime as $\eps\to0_+$, we obtain that
\[
\eps\,\frac d{dt}\,\mathsf{H}[F]\le-\,\mathsf{D}[F]\le-\frac{\lambda_M}{3\,(1+\lambda_M)}\,\frac{\lambda_m\,\lambda_M\,\eps}{(1+\lambda_M)\,C_M^2}\,\mathsf{D}[F]
\]
which proves that the estimate of Proposition~\ref{theo:DMS} becomes
\[
\lambda\ge\frac{\lambda_m\,\lambda_M^2}{3\,(1+\lambda_M)^2\,C_M^2}\,.
\]
We observe that this rate is independent of $\eps$.

%%%%%%%%%%%%%%%%%%%%%%%%%%%%%%%%%%%%%%%%%%%%%%%%%%%%%%%%%%%%%%%%%%%%%%
\subsection{Decay rates based on Nash's inequality in the whole space case}
In the proof of Theorem~\ref{th:whole-space-mode}, $\overline\sigma$ has to be replaced by $\overline\sigma/\eps$ and in the limit as $\eps\to0_+$, we get that $\mathsf b\sim4\,\overline\sigma/\eps$ and~\eqref{Conditions} is satisfied with $4\,\mathsf a=\delta\sim\frac{\lambda_m}{8\,\overline\sigma^2}\,\eps$. Hence~\eqref{eq:finalestimate} asymptotically becomes, as $\eps\to0_+$,
\[
-\,\frac d{dt}\mathsf{H}[f]\ge\frac{\lambda_m}{4\,\overline\sigma^2}\,\mathsf c\,\Big(\tfrac2{1+\delta}\,\mathsf{H}[f]\Big)^{1+\frac2d}\,,
\]
which again gives a rate of decay which is independent of $\eps$. The
algebraic decay rate in Theorem~\ref{th:whole-space-mode} is the
one of the heat equation on $\R^d$ and it is independent of $\eps$ in
the limit as $\eps\to0_+$.

%%%%%%%%%%%%%%%%%%%%%%%%%%%%%%%%%%%%%%%%%%%%%%%%%%%%%%%%%%%%%%%%%%%%%%
\subsection{Decay rates in the whole space case for distribution functions with moment cancellations}
The improved rate of Theorem~\ref{th:whole-space-zero-average} is consistent with a parabolic rescaling: if $f$ solves~\eqref{eq:model}, then $f^\eps(t,x,v)=\eps^{-d}\,f\(\eps^{-2}\,t,\eps^{-1}\,x,v\)$ solves~\eqref{rescaled}. With the notations of Section~\ref{sec:impr-decay-rate}, let $g^\eps=f^\eps-f_{\kern-0.5pt\bullet}^\eps\,\varphi(\cdot/\eps)$, with $\varphi^\eps=\eps^{-d}\,\varphi(\cdot/\eps)$. The Fourier transform of~$g^\eps$ solves
\[
\eps^2 \partial_t \hat g^\eps+\eps\mathsf T \hat g^\eps=\mathsf L \hat g^\eps-\eps\,f_{\kern-0.5pt\bullet}^\eps \mathsf T\hat{\varphi^\eps}\,.
\]
The decay rate $\lambda$ in~\eqref{eq:decayF} becomes $\lambda/\eps^2$ and the decay rate of the semi-group generated by $\mathsf L-\eps \mathsf T$ is, with the notations of Corollary~\ref{lec14-modelemm}, $\mu_{\eps\xi}$. Moreover, $\Lambda$ in~\eqref{eq:mu_xi} is given by $\Lambda=\frac13\,\min\big\{1,\Theta\big\}$ for any $\eps>0$, small enough. Duhamel's formula~\eqref{Duhamel} has to be replaced by
\begin{multline*}
\left\|\hat g^\eps(t,\xi,\cdot)\right\|_{\mathrm L^2\(d\gamma_k\)}\le C\,e^{-\frac{\mu_{\eps\xi}}{2\eps^2} t}\left\|\hat g_0^\eps(\xi,\cdot)\right\|_{\mathrm L^2\(d\gamma_k\)}\\
+C\int_0^te^{-\frac{\mu_{\eps\xi}}{2\,\eps^2}\,(t-s)}\left\|f_{\kern-0.5pt\bullet}^\eps(s,\cdot)\right\|_{\mathrm L^2\(|v|^2\,d\gamma_k\)}|\eps\,\xi|\,|\hat\varphi(\eps\,\xi)|\,ds\,.
\end{multline*}
Using $\lim_{\eps\to0_+}\frac{\mu_{\eps\xi}}{\eps^2}=\lim_{\eps\to0_+}\frac{\Lambda|\,\xi|^2}{1+\eps^2|\xi|^2}=\Lambda|\,\xi|^2$, a computation similar to the one of Section~\ref{sec:impr-decay-rate} shows that the first term of the r.h.s. is estimated by
\begin{align*}
&\int_{\R^d}e^{-\frac{\mu_{\eps\xi}}{\eps^2}\,t}\,\left\|\hat g_0^\eps(\xi,\cdot)\right\|_{\mathrm L^2\(d\gamma_k\)}^2\,d\xi\\&=\int_{|\xi|\le\frac1{\eps}}\,e^{-\frac{\mu_{\eps\xi}}{\eps^2}\,t}\,\left\|\hat g_0^\eps(\xi,\cdot)\right\|_{\mathrm L^2(d\gamma_k)}^2\,d\xi+\int_{|\xi|>\frac1{\eps}}\,e^{-\frac{\mu_{\eps\xi}}{\eps^2}\,t}\,\left\|\hat g_0^\eps(\xi,\cdot)\right\|_{\mathrm L^2(d\gamma_k)}^2\,d\xi\\
&\le\left\|g_0^\eps\right\|_{\mathrm L^2\(d\gamma_k;\,\mathrm L^1(|x|\,dx)\)}^2\int_{\R^d}|\xi|^2\,e^{-\frac{\Lambda}2\,|\xi|^2\,t}\,d\xi+\left\|g_0^\eps\right\|_{\mathrm L^2(dx\,d\gamma_k)}^2\,e^{-\frac{\Lambda}2\frac t{\eps^2}}\,,
\end{align*}
while the square of the second term is bounded by
\begin{multline*}
\left\|f_{\kern-0.5pt\bullet}^\eps(t=0,\cdot)\right\|_{\mathrm L^2\(|v|^2\,d\gamma_k\)}^2\int_{\R^d}|\,\eps\,\xi|^2\,|\hat\varphi(\eps\,\xi)|^2\(\int_0^{\eps^{-2}\,t}e^{-\frac12\,\mu_{\eps\xi}\(\eps^{-2}t-s\)}\,e^{-\frac12\,\lambda\,s}\,ds\)^2\,d\xi\\
\le\left\|f_0\right\|_{\mathrm L^2\(|v|^2\,d\gamma_k;\,\mathrm L^1(dx)\)}^2\(C_1\,\tfrac{\eps^{d+1}}{t^{\frac d2+1}}+\frac{C_2}{\eps^3}\,e^{-\min\lbrace\frac\Lambda2,\lambda\rbrace\,\frac t{\eps^2}}\)\,.
\end{multline*}
By collecting all terms and using Plancherel's formula, we conclude that the rate of convergence of Theorem~\ref{th:whole-space-zero-average} applied to the solution of~\eqref{rescaled} is independent of~$\eps$. We also notice that the scaled spatial density $\rho_{f^\eps}=\int_{\R^d}f^\eps\,dv$ satisfies
\[
\left\|\rho_{f^\eps}(t,\cdot)\right\|_{\mathrm L^2\(dx\)}^2\le\frac{\mathcal C_0}{(1+t)^{1+\frac d2}}\quad\forall\,t\ge0
\]
for some positive constant $\mathcal C_0$ which depends on $f_0$ but is independent of $\eps$. This is the decay of the heat equation with an initial datum of zero average. 

Similar estimates can be obtained in the framework of Theorem~\ref{th:whole-space-zero-average2}.

%%%%%%%%%%%%%%%%%%%%%%%%%%%%%%%%%%%%%%%%%%%%%%%%%%%%%%%%%%%%%%%%%%%%%%
%%%%%%%%%%%%%%%%%%%%%%%%%%%%%%%%%%%%%%%%%%%%%%%%%%%%%%%%%%%%%%%%%%%%%%
\section*{Acknowledgments} {\small This work has been partially supported by the Projects EFI (E.B., J.D., ANR-17-CE40-0030), Kibord (E.B., J.D., ANR-13-BS01-0004) and STAB (J.D., ANR-12-BS01-0019) of the French National Research Agency (ANR). The work of C.S. has been supported by the Austrian Science Foundation (grants no. F65 and W1245), by the Fondation Sciences
Math\'ematiques de Paris, and by Paris Science et Lettres. C.M. and E.B. acknowledge partial funding by the ERC grants MATKIT 2011-2016 and MAFRAN 2017-2022. Moreover C.M. and C.S. are very grateful for the hospitality at 
Universit\'e Paris-Dauphine.}\\
{\sl\scriptsize\copyright~2019 by the authors. This paper may be reproduced, in its entirety, for non-commercial purposes.}

%%%%%%%%%%%%%%%%%%%%%%%%%%%%%%%%%%%%%%%%%%%%%%%%%%%%%%%%%%%%%%%%%%%%%%
%\bibliographystyle{acm}
%\bibliography{BDMMS}

\def\cprime{$'$} \def\cprime{$'$} \def\cprime{$'$} \def\cprime{$'$}
\begin{thebibliography}{10}

\bibitem{achleitner2016linear}
{\sc Achleitner, F., Arnold, A., and Carlen, E.~A.}
\newblock On linear hypocoercive {BGK} models.
\newblock In {\em From Particle Systems to Partial Differential Equations III}.
  Springer, 2016, pp.~1--37.

\bibitem{bakry:hal-00634523}
{\sc Bakry, D., Barthe, F., Cattiaux, P., and Guillin, A.}
\newblock {A simple proof of the Poincar{\'e} inequality for a large class of
  probability measures including the log-concave case}.
\newblock {\em {Electronic Communications in Probability} 13\/} (2008), 60--66.

\bibitem{Bartier201176}
{\sc Bartier, J.-P., Blanchet, A., Dolbeault, J., and Escobedo, M.}
\newblock Improved intermediate asymptotics for the heat equation.
\newblock {\em Applied Mathematics Letters 24}, 1 (2011), 76 -- 81.

\bibitem{MR954373}
{\sc Beckner, W.}
\newblock A generalized {P}oincar\'{e} inequality for {G}aussian measures.
\newblock {\em Proc. Amer. Math. Soc. 105}, 2 (1989), 397--400.

\bibitem{bhatnagar1954model}
{\sc Bhatnagar, P.~L., Gross, E.~P., and Krook, M.}
\newblock A model for collision processes in gases. {I}. small amplitude
  processes in charged and neutral one-component systems.
\newblock {\em Physical review 94}, 3 (1954), 511.

\bibitem{MR1200643}
{\sc Bouchut, F.}
\newblock Existence and uniqueness of a global smooth solution for the
  {V}lasov-{P}oisson-{F}okker-{P}lanck system in three dimensions.
\newblock {\em J. Funct. Anal. 111}, 1 (1993), 239--258.

\bibitem{BouinHoffmann}
{\sc Bouin, E., Hoffmann, F., and Mouhot, C.}
\newblock Exponential decay to equilibrium for a fiber lay-down process on a
  moving conveyor belt.
\newblock {\em SIAM J. Math. Anal. 49}, 4 (2017), 3233--3251.

\bibitem{Caceres-Carrillo-Goudon}
{\sc C{\'a}ceres, M.~J., Carrillo, J.~A., and Goudon, T.}
\newblock Equilibration rate for the linear inhomogeneous relaxation-time
  {B}oltzmann equation for charged particles.
\newblock {\em Comm. Partial Differential Equations 28}, 5-6 (2003), 969--989.

\bibitem{DegGouPou}
{\sc Degond, P., Goudon, T., and Poupaud, F.}
\newblock Diffusion limit for nonhomogeneous and non-microreversible processes.
\newblock {\em Indiana Univ. Math. J. 49\/} (2000), 1175--1198.

\bibitem{Dolbeault2009511}
{\sc Dolbeault, J., Mouhot, C., and Schmeiser, C.}
\newblock Hypocoercivity for kinetic equations with linear relaxation terms.
\newblock {\em Comptes Rendus Math{\'e}matique 347}, 9-10 (2009), 511 -- 516.

\bibitem{DMS-2part}
{\sc Dolbeault, J., Mouhot, C., and Schmeiser, C.}
\newblock Hypocoercivity for linear kinetic equations conserving mass.
\newblock {\em Transactions of the American Mathematical Society 367\/} (2015),
  3807--3828.

\bibitem{Duan2011}
{\sc Duan, R.}
\newblock Hypocoercivity of linear degenerately dissipative kinetic equations.
\newblock {\em Nonlinearity 24}, 8 (2011), 2165--2189.

\bibitem{MR1969727}
{\sc Eckmann, J.-P., and Hairer, M.}
\newblock Spectral properties of hypoelliptic operators.
\newblock {\em Comm. Math. Phys. 235}, 2 (2003), 233--253.

\bibitem{2017arXiv170204168E}
{\sc {Evans}, J.}
\newblock Hypocoercivity in {P}hi-entropy for the linear {B}oltzmann equation
  on the torus.
\newblock {\em ArXiv e-prints\/} (Feb. 2017).

\bibitem{MR1379589}
{\sc Glassey, R.~T.}
\newblock {\em The {C}auchy problem in kinetic theory}.
\newblock Society for Industrial and Applied Mathematics (SIAM), Philadelphia,
  PA, 1996.

\bibitem{MR3779780}
{\sc Gualdani, M.~P., Mischler, S., and Mouhot, C.}
\newblock Factorization of non-symmetric operators and exponential
  {$H$}-theorem.
\newblock {\em M\'{e}m. Soc. Math. Fr. (N.S.)}, 153 (2017), 1--137.

\bibitem{MR3479064}
{\sc Han-Kwan, D., and L\'eautaud, M.}
\newblock Geometric analysis of the linear {B}oltzmann equation {I}. {T}rend to
  equilibrium.
\newblock {\em Ann. PDE 1}, 1 (2015), Art. 3, 84.

\bibitem{MR2215889}
{\sc H\'erau, F.}
\newblock Hypocoercivity and exponential time decay for the linear
  inhomogeneous relaxation {B}oltzmann equation.
\newblock {\em Asymptot. Anal. 46}, 3-4 (2006), 349--359.

\bibitem{MR2034753}
{\sc H\'erau, F., and Nier, F.}
\newblock Isotropic hypoellipticity and trend to equilibrium for the
  {F}okker-{P}lanck equation with a high-degree potential.
\newblock {\em Arch. Ration. Mech. Anal. 171}, 2 (2004), 151--218.

\bibitem{MR0222474}
{\sc H\"ormander, L.}
\newblock Hypoelliptic second order differential equations.
\newblock {\em Acta Math. 119\/} (1967), 147--171.

\bibitem{Iacobucci2017}
{\sc Iacobucci, A., Olla, S., and Stoltz, G.}
\newblock Convergence rates for nonequilibrium {L}angevin dynamics.
\newblock {\em Annales math{\'e}matiques du Qu{\'e}bec\/} (Oct 2017).

\bibitem{MR0168018}
{\sc Il$\prime$in, A.~M., and Has$\prime$minski\u\i, R.~Z.}
\newblock On the equations of {B}rownian motion.
\newblock {\em Teor. Verojatnost. i Primenen. 9\/} (1964), 466--491.

\bibitem{Kavian}
{\sc {Kavian}, O., and {Mischler}, S.}
\newblock The {F}okker-{P}lanck equation with subcritical confinement force.
\newblock {\em ArXiv e-prints\/} (Dec. 2015).

\bibitem{MR1057534}
{\sc Kawashima, S.}
\newblock The {B}oltzmann equation and thirteen moments.
\newblock {\em Japan J. Appl. Math. 7}, 2 (1990), 301--320.

\bibitem{MR1503147}
{\sc Kolmogoroff, A.}
\newblock Zuf\"allige {B}ewegungen (zur {T}heorie der {B}rownschen {B}ewegung).
\newblock {\em Ann. of Math. (2) 35}, 1 (1934), 116--117.

\bibitem{MisMou}
{\sc Mischler, S., and Mouhot, C.}
\newblock Exponential stability of slowly decaying solutions to the
  kinetic-{F}okker-{P}lanck equation.
\newblock {\em Arch. Ration. Mech. Anal. 221}, 2 (2016), 677--723.

\bibitem{2017arXiv170310504M}
{\sc {Monmarch{\'e}}, P.}
\newblock {A note on Fisher Information hypocoercive decay for the linear
  Boltzmann equation}.
\newblock {\em ArXiv e-prints\/} (Mar. 2017).

\bibitem{Mouhot-Neumann}
{\sc Mouhot, C., and Neumann, L.}
\newblock Quantitative perturbative study of convergence to equilibrium for
  collisional kinetic models in the torus.
\newblock {\em Nonlinearity 19}, 4 (2006), 969--998.

\bibitem{Nash}
{\sc Nash, J.}
\newblock Continuity of solutions of parabolic and elliptic equations.
\newblock {\em Amer. J. Math. 80\/} (1958), 931--954.

\bibitem{MR751959}
{\sc Reed, M., and Simon, B.}
\newblock {\em Methods of modern mathematical physics. {I}}, second~ed.
\newblock Academic Press, Inc. [Harcourt Brace Jovanovich, Publishers], New
  York, 1980.
\newblock Functional analysis.

\bibitem{2017arXiv170610034V}
{\sc {V{\'a}zquez}, J.~L.}
\newblock {Asymptotic behaviour methods for the Heat Equation. Convergence to
  the Gaussian}.
\newblock {\em ArXiv e-prints\/} (June 2017).

\bibitem{MR1052014}
{\sc Victory, Jr., H.~D., and O'Dwyer, B.~P.}
\newblock On classical solutions of {V}lasov-{P}oisson {F}okker-{P}lanck
  systems.
\newblock {\em Indiana Univ. Math. J. 39}, 1 (1990), 105--156.

\bibitem{MR2275692}
{\sc Villani, C.}
\newblock Hypocoercive diffusion operators.
\newblock In {\em International Congress of Mathematicians. Vol. III}. Eur.
  Math. Soc., Z\"urich, 2006, pp.~473--498.

\bibitem{Mem-villani}
{\sc Villani, C.}
\newblock {\em Hypocoercivity}.
\newblock Memoirs Amer. Math. Soc. 202, 2009.

\end{thebibliography}

\def\cprime{$'$} \def\cprime{$'$} \def\cprime{$'$} \def\cprime{$'$}

\bigskip
%%%%%%%%%%%%%%%%%%%%%%%%%%%%%%%%%%%%%%%%%%%%%%%%%%%%%%%%%%%%%%%%%%%%%%
\end{document}